\documentclass[11pt,letterpaper]{amsart}
\usepackage{mathtools, bbm} 

\usepackage{amsmath}
\usepackage{bm}
\usepackage{mathrsfs}
\usepackage{amssymb}
\usepackage{xcolor}
\usepackage{amscd} 
\usepackage{tikz}
\usepackage{graphicx}
\usepackage{subcaption}

\usepackage{hyperref}
\usepackage{cleveref} 
\usepackage{comment}
\usepackage[margin=1in]{geometry} 

\theoremstyle{plain}
 \newtheorem{thm}{\textbf{Theorem}}[section]
 \newtheorem{prop}[thm]{\textbf{Proposition}}
 \newtheorem{lem}[thm]{\textbf{Lemma}}
 \newtheorem{cor}[thm]{\textbf{Corollary}}
 \newtheorem{conj}[thm]{\textbf{Conjecture}}
\theoremstyle{definition}
 \newtheorem{ex}[thm]{\textbf{Example}}
 \newtheorem{dfn}[thm]{\textbf{Definition}}
\theoremstyle{remark}
 \newtheorem{rem}[thm]{\textbf{Remark}}
 \numberwithin{equation}{section}

\DeclareMathOperator{\rk}{rank}
\DeclareMathOperator{\ch}{ch}
\DeclareMathOperator{\Pch}{Pch}
\DeclareMathOperator{\inv}{inv}
\DeclareMathOperator{\ps}{ps}
\DeclareMathOperator{\Des}{Des}
\DeclareMathOperator{\Asc}{Asc}
\DeclareMathOperator{\Whit}{W\!H}

\DeclareMathOperator{\sym}{\mathfrak{S}}

\newcommand{\bQ}{\mathbb{Q}}

\newcommand{\bZ}{\mathbb{Z}}
\newcommand{\bF}{\mathbb{F}}

\newcommand{\boldvec}[1]{\underline{\bm{#1}}}

\newcommand\coveredby{\mathrel{\ooalign{$<$\cr
  \hidewidth\raise0.0ex\hbox{$\cdot\mkern2mu$}\cr}}}
\newcommand\covers{\mathrel{\ooalign{$>$\cr
  \hidewidth\raise0.0ex\hbox{$\cdot\mkern7mu$}\cr}}}

\renewcommand{\le}{\leqslant}
\renewcommand{\ge}{\geqslant}
\renewcommand{\setminus}{\smallsetminus}

 \author[Li and Sundaram]{Yifei Li and Sheila Sundaram}

\address{Yifei Li: University of Illinois at Springfield, Springfield, IL 62703  USA} 
\email{yli236@uis.edu}
 \address{Sheila Sundaram:  University of Minnesota, Minneapolis, MN 55455, USA}
\email{shsund@umn.edu}

\title[Boolean Segre powers]{Homology of Segre powers of Boolean and subspace  lattices}
\date{\today}

\begin{document}


\begin{abstract} Segre products of posets were defined by Bj\"orner and Welker (2005).
We investigate the homology representations of the $t$-fold Segre power $B_n^{(t)}$ of the Boolean lattice $B_n$.   The direct product $\sym_n^{\times t}$ of the symmetric group $\sym_n$  acts on the homology of  rank-selected subposets of $B_n^{(t)}$. We  give an explicit formula for the decomposition into $\sym_n^{\times t}$-irreducibles of the homology of the full poset, as well as  formulas for  the diagonal action of the symmetric group $\sym_n$. For  the rank-selected homology, we show that the stable principal specialisation of the product Frobenius characteristic of the $\sym_n^{\times t}$-module coincides with the corresponding rank-selected invariant of the $t$-fold Segre power of the subspace lattice. 
\end{abstract}

\maketitle

\section{Introduction}\label{sec:Intro}
 Let $B_n$ denote the Boolean lattice of subsets of an $n$-element set, and let $B_{n,q}$ denote the lattice of subspaces of an $n$-dimensional vector space over the finite field  $\bF_q$ with  $q$ elements. 

The \emph{Segre product}  of  posets   was first defined by Bj\"orner and Welker, who showed \cite[Theorem 1]{Segre_rees} that this operation preserves the property of being homotopy Cohen-Macaulay.  Let $P^{(t)}$ denote the $t$-fold Segre power $P\circ \cdots \circ P$ ($t$ factors) of a graded poset $P$. The  Segre square $P\circ P$ was studied by the first author  in \cite{YLiqCSV2023}, when $P$ is the Boolean lattice $B_n$ or the subspace lattice $B_{n,q}$.  Segre powers of the subspace lattice $B_{n,q}$ appear in an early paper of Stanley \cite[Ex. 1.2]{RPS-BinomialPosetsJCTA1976},  as an example of  binomial posets. See also \cite[Ex. 3.18.3]{ec1}.

The symmetric group $\sym_n$ acts on $B_n$, and hence the Segre power $B_n^{(t)}$ of $B_n$ carries two actions, one for the $t$-fold direct product $\sym_n^{\times t}$ of $\sym_n$ with itself, and the other for the symmetric group $\sym_n$.

 In this paper we study both these actions on the rank-selected subposets of the Cohen-Macaulay poset  $B_n^{(t)}$, giving  formulas for the irreducible decomposition  on the top homology of  $B^{(t)}_n$.  For the $t$-fold Segre power of the subspace lattice, 
a special case of a theorem of Stanley \cite[Theorem 3.1]{RPS-BinomialPosetsJCTA1976} shows 
that the M\"obius number of $B_{n,q}^{(t)}$ is given by 
$(-1)^nW^{(t)}_n(q)$ where 
$$ W^{(t)}_n(q)
:=\sum_{(\sigma_1,\ldots \sigma_t)\in  \sym_n^{\times t}}\, 
\prod_{i=1}^t q^{\inv(\sigma_i)}.$$ Here 
$\inv(\tau)$ is the number of inversions of the permutation $\tau$, 
and the sum is over all $t$-tuples  of permutations in $\sym_n$ with no common ascent, $i$ being an ascent of a permutation $\sigma$ if $\sigma(i)<\sigma(i+1)$.  When $q=1$ this specialises to   
the dimension of the homology of $B_n^{(t)}$; it is 
the number $w_n^{(t)}$ of $t$-tuples  of permutations in the symmetric group $\sym_n$ with no common ascent.  The numbers $w_n^{(2)}$ first appear in work of Carlitz, Scoville and Vaughn \cite{CSV}. 
For arbitrary $t$ the numbers $w_n^{(t)}$ have also already appeared in the literature; see Abramson and Promislow  \cite{AbramsonPromislowJCTA1978}. This suggests  a deeper connection between the homology modules of the Segre powers of the subspace lattice and the Boolean lattice.  

The paper is organised as follows. Prerequisites are reviewed in Section~\ref{sec:Segre-EL}. 
In order to describe the rank-selected homology modules of $B_n^{(t)}$ for the $\sym_n^{\times t}$-action, in  Section~\ref{sec:prod_Frob} 
we develop an extension of the \emph{product Frobenius characteristic} introduced in \cite{YLiqCSV2023}. In Section~\ref{sec:tfold-Boolean-lattice-repn}  we use  the Whitney homology technique of  \cite{SundaramAIM1994} to  derive recursive formulas for these representations involving symmetric functions in $t$ sets of variables.   Section~\ref{sec:rank-selection} extends these results to the action of $\sym_n^{\times t}$ on the   chains and  homology of all  rank-selected \cite{RPSGaP1982} subposets of $B^{(t)}_n$. Finally in Section~\ref{sec:stable-ps} we use these formulas to investigate the stable principal specialisations of the rank-selected representations of $B^{(t)}_n$, and establish a connection with the corresponding rank-selected invariants of the $t$-fold Segre power $B_{n,q}^{(t)}$ of the subspace lattice  $B_{n,q}$.   For the top homology this fact was established in \cite{YLiqCSV2023}, in the case $t=2$.

Our framework allows us to obtain explicit formulas for the homology representation of $B_n^{(t)}$.
A key feature of these formulas is Definition~\ref{def:map-Phi-S}, where we introduce an injective algebra homomorphism 
  $\Phi_t:\Lambda_n(x) \rightarrow \otimes_{j=1}^t \Lambda_n(X^j)$  from the algebra of  symmetric functions of homogeneous degree $n$ in a single set of variables, to the tensor product of the algebras of degree $n$-symmetric functions in $t$ sets of variables.  We show that $\Phi_t$ maps the elementary symmetric function $e_n$ to the product Frobenius characteristic $\beta^{(t)}_n$ of the top homology of $B^{(t)}_n$.    By exploiting properties of the homomorphism  $\Phi_t$, we obtain the main results of this paper:
  \begin{enumerate}
      \item Theorem~\ref{thm:Bnt-irreps}
        gives  the decomposition into irreducibles of the top homology  $ \tilde{H}_{n-2}(B_n^{(t)})$ of  $B_n^{(t)}$ under the action of $\sym_n^{\times t}$.
      \item Theorem~\ref{thm:Sn-diag-action-homology-Bn-t}
       gives a formula for the irreducible decomposition of the diagonal $\sym_n$-action on $ \tilde{H}_{n-2}(B_n^{(t)})$ in terms of Kronecker products, including an explicit formula for the character values.
      \item  Theorem~\ref{thm:rank-selected-homology-Bnt} gives a recursive formula for the product Frobenius characteristic of the rank-selected homology, from which one can obtain  explicit formulas for the irreducible decomposition. 
      \item Theorem~\ref{thm:ps-rank-selection}  shows that the stable principal specialisation of the product Frobenius characteristic of the rank-selected homology of $B^{(t)}_n$ gives, up to a factor,  the corresponding rank-selected invariant for $B_{n,q}^{(t)}$.   
  \end{enumerate}
  All homology  in this work is reduced, and  taken with rational coefficients.

\vskip.1in
\noindent 
{\bf Acknowledgment} This material is based upon work supported by the National Science Foundation under Grant No. DMS-1928930, while the authors were in residence at the Simons Laufer Mathematical Sciences Research Institute in Berkeley, California, during the  Summer of 2023.  The authors thank John Shareshian, whose  suggestion to extend the results of \cite{YLiqCSV2023} to $t$-fold Segre powers led to the discoveries in this paper.  They are  grateful for the comments of the anonymous referee.  Thanks also to  the referees  of FPSAC 2025 (Sapporo), where an Extended Abstract  \cite{LiSu2025FPSAC}  of this paper  has been accepted.

\section{Segre powers and rank-selected invariants}\label{sec:Segre-EL}

We refer the reader to \cite{BjTopMeth1995, ec1, WachsPosetTop2007} for background on posets and topology.

Recall \cite{ec1} that the product poset $P\times Q$ of two posets $P,Q$ has order relation defined by $(p,q)\le (p',q') $ if and only if $p\le_P p'$ and $q\le_Q q'$. Segre products are defined in greater generality by Bj\"orner and Welker  in \cite{Segre_rees}. This paper is concerned with the following special case. 
\begin{dfn}[\cite{Segre_rees}]\label{def:Segre-prod} 
 Let $P$ be a bounded graded poset. The $t$-fold Segre power ${P\circ \cdots \circ P}$ ($t$ factors), denoted $P^{(t)}$,  is  defined  for all $t\ge 2$ to be the induced subposet of the $t$-fold product poset $P\times \cdots \times P$ ($t$ factors) consisting of $t$-tuples $(x_1,\ldots, x_t)$ such that $\rk(x_i)=\rk(x_j), 1\le i,j\le t$.  The cover relation in $P^{(t)}$ is thus 
$(x_1,\ldots,x_t) \coveredby (y_1,\ldots, y_t)$ if and only if $x_i\coveredby y_i$ in $P$, for all $i=1,\ldots, t$.   When $t=1$ we set $P^{(1)}$ equal to $P$.
\end{dfn}
It follows that $P^{(t)}$ is also a ranked poset which inherits the rank function of $P$. 

 Figure~\ref{fig:Segre square} shows the Segre square $P\circ P$ of a poset $P$,  an induced subposet of  $P\times P$. 
Missing in  $P\circ P$ are   these elements in the product $P\times P$:
$(a,c), (a,d), (b,c), (b,d), (c,a), (c,b),$ $ (d,a), (d,b)$, 
as well as all $(\hat 0, y), (y, \hat 0), y\ne \
\hat 0$, and $(x, \hat 1), (\hat 1, x)$, $x\ne \hat 1$. 

\begin{figure}\
    \begin{subfigure}{0.3\textwidth}
        \centering
        \begin{tikzpicture}[scale=0.8]
        \node (max) at (0,0) {$\hat{1}$};
        \node (c) at (-1,-1) {$c$};
        \node (d) at (1,-1) {$d$};
        \node (a) at (-1,-2) {$a$};
        \node (b) at (1,-2) {$b$};
        \node (min) at (0,-3) {$\hat{0}$};
        \draw (min)--(a)--(c)--(max)--(d)--(a);
        \draw (min)--(b)--(d);
    \end{tikzpicture}
    \caption{Hasse diagram of $P$. }
    \label{fig:P-repeat-labels}
    \end{subfigure}
    \begin{subfigure}{0.5\textwidth}
        \centering
        \begin{tikzpicture}[scale=0.8]
        \node (max) at (0,0) {$\hat{1}, \hat 1$};
        \node (cc) at (-3,-1) {$c,c$};
        \node (cd) at (-1,-1) {$c,d$};
        \node (dc) at (1,-1) {$d,c$};
        \node (dd) at (3,-1) {$d,d$};
        \node (aa) at (-3,-3) {$a,a$};
        \node (ab) at (-1,-3) {$a,b$};
        \node (ba) at (1,-3) {$b,a$};
        \node (bb) at (3,-3) {$b,b$};
        \node (min) at (0,-4) {$\hat{0}, \hat 0$};
        \draw (min)--(aa)--(cc)--(max)--(cd)--(aa)--(dc)--(max);
        \draw (aa)--(dd);
        \draw (ab)--(cd);
        \draw (ba)--(dc);
        \draw  (min)--(ab)--(dd)--(max);
        \draw  (min)--(ba)--(dd);
        \draw (min)--(bb)--(dd);
    \end{tikzpicture}
    \caption{Hasse diagram of $P\circ P$}
    \label{fig:PP-repeat-labels}
    \end{subfigure}
\caption{$P\circ P$ is an induced subposet of the product poset $P\times P$.}
\label{fig:Segre square}
\end{figure}

Let $P$ be a finite graded bounded poset of rank $n$, and let $J\subset [n-1]=\{1,\ldots,n-1\}$ be any subset of nontrivial ranks.  Let $P(J)$ denote  the rank-selected bounded subposet of $P$ consisting of elements in the rank-set $J$, together with $\hat 0$ and $\hat 1$.
Stanley  \cite[Section 3.13]{ec1} defined two rank-selected invariants $\tilde\alpha_P(J)$ and $\tilde\beta_P(J)$ 
as follows.  
\begin{itemize}
    \item
$\tilde\alpha_P(J)$ is the number of maximal chains in the rank-selected subposet $P(J)$, and 
\item
$\tilde\beta_P(J)$ is the integer defined by the equation 
\[\tilde\beta_P(J):=\sum_{U\subseteq J} (-1)^{|J|-|U|}\tilde\alpha_P(U).\]  
Equivalently, 
\[\tilde\alpha_P(J)=\sum_{U\subseteq J} \tilde\beta_P(U).\]
\end{itemize}
One also has the formula for the M\"obius number of the rank-selected subposet $P(J)$ \cite[Eqn. (3.54)]{ec1}:
\begin{equation}\label{eqn:rank-sel-inv-to-mu}\tilde\beta_P(J)=(-1)^{|J|-1} \mu_{P(J)}(\hat 0, \hat 1).\end{equation}

When the poset $P$ has the recursive structure  described in the lemma below, the rank-selected invariants satisfy a pleasing recurrence that we record for later use in  Section~\ref{sec:stable-ps}.

\begin{lem}\label{lem:rec-rank-select-Betti-GENERAL} Let $P$ be a  graded, bounded poset  of rank $n$, with the property that for any $x\in P$, the poset structure of the interval $(\hat 0,  x)$ depends only on the  rank of $x$.   Thus we may write $P_i=(\hat 0, x_0)$ for any $x_0$ of rank $i$. Let $wh_i(P)$ denote the number of elements of $P$ at rank $i$.  Then we have the following recurrence for the rank-selected invariants $\tilde\beta_P(J)$ of $P$, $J=\{1\le j_1<\cdots<j_r\le n-1\}\subseteq [n-1]$. 
\begin{equation}\label{eqn:actual-rec-mu-rank-select-GENERAL} 
\tilde\beta_P(J) +\tilde\beta_P(J\setminus\{j_r\}) =wh_{j_r}(P) \cdot \tilde\beta_{P_{j_r}}(J\setminus\{j_r\}  )
\end{equation}

For the full poset $P$, we have 
\begin{equation}\label{eqn:full-mu-rec-GENERAL} \mu_P(\hat 0,\hat 1)=-\sum_{i=0}^{n-1} wh_i(P)\cdot \mu_{P_i}(\hat 0, \hat 1).\end{equation}
\end{lem}
\begin{proof} Observe first that the condition satisfied by $P$ is inherited by every rank-selected subposet $Q$ of $P$.  For such  posets $Q$ of rank $k$, the M\"obius function recurrence gives 
\begin{align}\label{eqn:mu-uniform-poset}
\mu_Q(\hat 0, \hat 1)&=-\sum_{x: \rk(x)=k-1}\mu_Q(\hat 0, x) -\sum_{x: \rk(x)\le k-2}\mu_Q(\hat 0, x)\\ \notag
&=-|\{x:\rk(x)=k-1\}|\cdot \mu_Q(\hat 0, x_0) +\mu_{Q_{\{\le k-2\}}}(\hat 0, \hat 1),\\
&\text{where $x_0$ is any fixed element of rank $k-1$}. \notag
\end{align}
Here $Q_{\{\le k-2\}}$ is the subposet of $Q$ consisting of the bottom $k-2$ ranks.

Applying this to $Q=P(J)$ now gives the result.

By refining the M\"obius function recurrence according to the rank, we obtain the special case for the full poset $Q=P$, since 
\[\mu_P(\hat 0,\hat 1)=-\sum_{i=0}^{n-1} \sum_{\text{ $x$ at rank $i$}} \mu_P(\hat 0, x). \qedhere\]
\end{proof}

The condition of the lemma is satisfied by the Boolean lattice $B_n$ and also by the subspace lattice $B_{n,q}$.
We can now derive a recurrence for the rank-selected invariants $\tilde\beta_{B_{n,q}^{(t)}}(J)$, $J\subseteq [n-1]$,  
of the $t$-fold Segre power $B_{n,q}^{(t)}$. From \eqref{eqn:rank-sel-inv-to-mu}, these are also the unsigned M\"obius numbers of the corresponding rank-selected subposets.  

Recall that the number of $i$-dimensional subspaces of the $n$-dimensional vector space  $\mathbb{F}^n_q$ \cite[Proposition 1.7.2]{ec1} is given by the $q$-binomial coefficient

\begin{equation}\label{eqn:q-bin-REMOVE}{n\brack i}_q:= \frac{(1-q^n) (1-q^{n-1})\cdots (1-q)}{(1-q^i) (1-q^{i-1})\cdots (1-q)\,(1-q^{n-i}) (1-q^{n-i-1})\cdots (1-q)  }. \end{equation} 
Let $(-1)^{n-2}W^{(t)}_n(q)$ be the M\"obius number of  the $t$-fold Segre power of the subspace lattice. When $q=1$, $B_{n,q}$ specialises to $B_n$, and hence from Proposition~\ref{prop:mu-Bnt} we have $W^{(t)}_n(1)=w_n^{(t)}$.  

To avoid a profusion of parentheses,  for the $k$th power of the $q$-binomial coefficient we write ${n\brack i}_q^k$. 

\begin{prop}\label{prop:rec-rank-select-Betti-Bnq} We have the following recurrence for the rank-selected invariants of $B_{n,q}^{(t)}$.
For the rank-set  $J=\{1\le j_1<\cdots<j_r\le n-1\}$:
\begin{equation}\label{eqn:actual-rec-mu-rank-select-Bnq} 
\tilde\beta_{B_{n,q}^{(t)}}(J) +\tilde\beta_{B_{n,q}^{(t)}}(J\setminus\{j_r\}) ={n \brack j_r}_q^t \tilde\beta_{B_{j_r}^{(t)}(q)}(J\setminus\{j_r\}  )
\end{equation}

For the full poset $B_{n,q}^{(t)}$, we have the recurrence 
\begin{equation}\label{eqn:rec-mu-tfold-subspacelattice} 
W^{(t)}_n(q)=\sum_{i=0}^{n-1} (-1)^{n-1-i} {n\brack i}_q^t W^{(t)}_i(q).
\end{equation}
\end{prop}
\begin{proof}  We apply Lemma~\ref{lem:rec-rank-select-Betti-GENERAL} to the rank-selected subposet $B_{n,q}^{(t)}(J)$, using the fact that $wh_i$, the number of elements at rank $i$ in the $t$-fold Segre power, is precisely ${n\brack i}^t_q$. Also,  if $x_0$ is at rank $j_r$,  the interval $(0,x_0)$ in $B_{n,q}^{(t)}(J)$ is poset isomorphic to the rank-selected subposet of $B_{j_r}^{(t)}(q)$ corresponding to the rank-set $J\setminus\{j_r\}$. Using \eqref{eqn:rank-sel-inv-to-mu} for the passage from M\"obius numbers to rank-selected invariants, the first recurrence in Lemma~\ref{lem:rec-rank-select-Betti-GENERAL} now gives \eqref{eqn:actual-rec-mu-rank-select-Bnq}. 

Similarly, the second recurrence in Lemma~\ref{lem:rec-rank-select-Betti-GENERAL} gives \eqref{eqn:rec-mu-tfold-subspacelattice}. \end{proof}

The recurrence~\eqref{eqn:actual-rec-mu-rank-select-Bnq}, in conjunction with an equivariant version of the recurrence in Lemma~\ref{lem:rec-rank-select-Betti-GENERAL} for the Boolean lattice that we derive in Section~\ref{sec:rank-selection}, will be used in Section~\ref{sec:stable-ps} when we consider the stable principal specialisation.

 We conclude this section by explaining more precisely the relevance of the numbers $w_n^{(t)}$ mentioned  in the Introduction.  We need  
 the following expression due to Stanley for $\tilde\beta_{B_{n,q}^{(t)}}(J)$ as a polynomial in $q$ with nonnegative coefficients.  Let $\Asc(\sigma)$ denote the ascent set of $\sigma$, that is, the set $\{i: 1\le i\le n-1, \sigma(i)>\sigma(i+1)\}$ of ascents of $\sigma$. 

\begin{thm}[{\cite[Theorem~3.1]{RPS-BinomialPosetsJCTA1976}}] \label{thm:rank-sel-beta-inv-N-of-q} Let $J\subseteq [n-1]$.  Write $J^c=[n-1]\setminus J$. Then 
for the rank-selected $t$-fold Segre power of the subspace lattice $B_{n,q}^{(t)}$, one has 
\[\tilde\beta_{B_{n,q}^{(t)}}(J)
=\sum_{\substack
{(\sigma^1,\ldots,\sigma^t)\in\mathfrak{S}_n^{\times t}\\
{J^c=\cap_{i=1}^t \Asc(\sigma^i)}
}} \prod_{i=1}^t q^{\inv(\sigma^i)}\]

In particular, 
%
the M\"obius number of the $t$-fold Segre power of the subspace lattice is $(-1)^{n-2}W^{(t)}_n(q)$, where 
\[ W^{(t)}_n(q)
:=\sum_{(\sigma_1,\ldots \sigma_t)\in \sym_n^{\times t}}\,
\prod_{i=1}^t q^{\inv(\sigma_i)},\] 
and the sum is over all $t$-tuples of permutations in $\sym_n$ with no common ascent.
\end{thm}

Setting $q=1$ gives the  special case of the Segre powers of the Boolean lattice, as mentioned  in the Introduction.
The generating function below  appears in \cite{RPS-BinomialPosetsJCTA1976}.  The numbers $w_n^{(t)}$ also appear  in \cite{AbramsonPromislowJCTA1978}.

\begin{prop}[{See \cite[Eqn. (28) and Theorem 3.1]{RPS-BinomialPosetsJCTA1976}}] \label{prop:mu-Bnt} The M\"obius number of $B^{(t)}_n$ is given by $(-1)^{n} w^{(t)}_n$, where for $n\ge 1$, 
$w^{(t)}_n$ is the number of $t$-tuples of permutations in $\sym_n$ with no common ascent.  Hence, setting $w_0^{(t)}=1$, the numbers $w^{(t)}_n$ satisfy the recurrence 
\begin{equation}\label{eqn:dim-Segre-homology}
 \sum_{i=0}^{n} (-1)^{i}   w^{(t)}_i\binom{n}{i}^t=0.
\end{equation}
Furthermore, we have the generating function  
\[\sum_{n\ge 0} w^{(t)}_n \frac{z^n}{n!^t}=\frac{1}{f(z)},
\quad \text{ where } f(z)=\sum_{n\ge 0} (-1)^n \frac{z^n}{n!^t}.\]
More generally, for the rank selection $J\subseteq [n-1]$, the M\"obius number  $\mu(B^{(t)}_n(J) )$ of $B^{(t)}_n(J)$ is given by $(-1)^{|J|-1} w^{(t)}_n(J)$, where 
$w^{(t)}_n(J)$ is the number of $t$-tuples of permutations in $\sym_n$ such that their set of common ascents coincides with the complement of $J$ in $[n-1]$.
\end{prop}
\begin{proof} The first recurrence is a restatement of the M\"obius function recurrence for the lattice $B^{(t)}_n$, and the generating function then follows. 
The statement for the rank-selected M\"obius number is the case $q=1$ of \cite[Theorem 3.1]{RPS-BinomialPosetsJCTA1976}.
\end{proof}
\section{The product Frobenius characteristic}\label{sec:prod_Frob}

We refer to \cite{Macd1995} for all background on symmetric functions and representations of the symmetric group $\mathfrak{S}_n$. See also \cite[Chapter 7]{ec2}.  In particular, $h_n$, $e_n$ and $p_n$  are respectively  the homogeneous, elementary, and power sum symmetric functions of degree $n$, giving rise to basis elements $h_\lambda, e_\lambda$ and $p_\lambda$ indexed by partitions $\lambda$ of $n$, in the algebra of symmetric functions of homogeneous degree $n$, and $s_\lambda$ is the Schur function indexed by  $\lambda$.

The action of the symmetric group $\sym_n$ on the Boolean lattice $B_n$ extends naturally to an action of the $t$-fold direct product $\sym_n^{\times t}:=\sym_n\times\cdots\times \sym_n$ ($t$ factors),  on the $t$-fold Segre power of $B_n$.  For the Segre square $B_n\circ B_n$, a \emph{product Frobenius map}, generalizing the well-known ordinary Frobenius characteristic in \cite{Macd1995}, was defined \cite{YLiqCSV2023} in order to study the action of $\sym_n\times \sym_n$ on the homology.

The goal of this section is to rigorously define and extend this construction to arbitrary $t$. We  begin by giving an alternative description of the product Frobenius map in \cite{YLiqCSV2023}. 
As in \cite[Chapter 1, Section  7]{Macd1995}, let $R^n$ denote the vector space spanned by the irreducible characters  of the symmetric group $\sym_n$ over $\bQ$, or equivalently the vector space spanned by the class functions of $\sym_n$.
Let $R=\oplus_{n\ge 0} R^n$. Then $R$ is equipped with the structure of a graded commutative and associative ring with identity element 1 for the group $\sym_0=\{1\}$, arising from the bilinear map $R^m\times R^n\rightarrow R^{m+n}$, defined by $(f,g)\mapsto (f\times g)\uparrow_{\sym_m\times\sym_n}^{\sym_{m+n}}$, the induced character from  $f$ and $g$.

Let $\Lambda^m(X)$ be the ring of symmetric functions in the set of variables $X$, of homogeneous degree $m$, and let $\Lambda(X)=\oplus_{m\ge 0} \Lambda^m(X)$. 
Write $\mu\vdash n$ for an integer partition $\mu=(\mu_1\ge\cdots\ge \mu_\ell)$ of the integer $n\ge 1$, so that $\sum_{i=1}^\ell \mu_i=n$, $\mu_i\ge 1$ for all $i$, and $\ell(\mu)$ for the number of parts $\mu_i$ of $\mu$. (There is only one integer partition of $0$, the empty partition  with zero parts.)

The ordinary Frobenius characteristic map $\ch$ is defined as follows. For each $f\in R^n$, 
\begin{equation}\label{eqn:ordinary-charmap-p}
\ch(f):=\sum_{\mu\vdash n} z_\mu^{-1} f_\mu p_\mu(X),
\end{equation}
where $f_\mu$ denotes the value of the class function $f$ on the class indexed by the partition $\mu\vdash n$,  $z_\mu$ is the order of the centraliser of a permutation of type $\mu$ in the symmetric group $\sym_n$, and $p_\mu(X)$ is the power sum symmetric function indexed by $\mu$.
In particular when $f$ is the irreducible character $\chi^\lambda$ indexed by the partition $\lambda$, then 
\begin{equation}\label{eqn:ordinary-charmap-s}
\ch(\chi^\lambda)=s_\lambda(X),
\end{equation}
where $s_\lambda(X)$ is the Schur function indexed by $\lambda$. The set of Schur functions $\{s_\lambda(X): \lambda\vdash n\}$ forms a basis for $\Lambda^n(X)$.  
Furthermore, $\ch$ is a ring isomorphism from $R$ to $\Lambda(X)$, since for $f\in R^m, g\in R^n$, 
\begin{equation}\label{eqn:ord-induction}
\ch\left((f\times g)\uparrow_{\sym_m\times\sym_n}^{\sym_{m+n}}\right) =\ch(f) \ch(g).
\end{equation}
In particular, for $\lambda\vdash m, \mu\vdash n$, 
\[\ch\left((\chi^\lambda\times \chi^\mu)\uparrow_{\sym_m\times\sym_n}^{\sym_{m+n}}\right)=s_\lambda(X) s_\mu(X). \]

  We wish to generalize this to a $t$-fold direct product of symmetric groups. The main idea originates in \cite{YLiqCSV2023}, where  the case $t=2$ was treated.  We will give a slightly different treatment here, elaborating on details that were omitted in \cite{YLiqCSV2023}. 
  Let $\boldvec{n}=(n_1,\ldots,n_t)\in\bZ_{\ge 0}^t$ be a $t$-tuple of nonnegative integers, and let  $\sym_{\boldvec{n}}$ be the direct product of symmetric groups $\bigtimes_{i=1}^t\sym_{n_i}$. The irreducible characters of $\sym_{\boldvec{n}}$ are indexed by $t$-tuples of partitions ${\boldvec{\lambda}}=(\lambda^1, \ldots, \lambda^t)$ where $\lambda^i\vdash n_i$. Let $R^{\boldvec{n}}$ denote the vector space spanned by the irreducible characters, or equivalently the vector space spanned by the class functions,  of the direct product of symmetric groups  $\sym_{\boldvec{n}}$ over $\bQ$.  Then $R^{\boldvec{n}}=\otimes_i R^{n_i}$.  
Let $\underline{R}=\oplus_{\boldvec{n}\in \bZ_{\ge 0}^t} R^{\boldvec{n}}$.

Let $(X^i)$, $i=1,\ldots, t$ be $t$ sets of variables. For each $i$ we consider the ring of symmetric functions $\Lambda^{n_i}(X^i)$ in the variables $(X^i)$, of homogeneous degree $n_i$. 
As in \cite[Chapter 1, Section 5, Ex. 25]{Macd1995}, we identify the tensor product $\bigotimes_{i=1}^t\Lambda^{n_i}(X^i)$ with products of functions of $t$ sets of variables $(X^i)_{i=1}^t$, symmetric in each set separately, i.e., with the vector space spanned by the set of elements 
\[\left\lbrace\prod_{i=1}^t f_{n_i}(X^i) : f_{n_i}(X^i)\in \Lambda^{n_i}(X^i)\right\rbrace.\]
Thus $\bigotimes_{i=1}^t f_{n_i}(X^i) \mapsto \prod_{i=1}^t f_{n_i}(X^i)$.

\begin{dfn}[{cf. \cite[Definition 3.2]{YLiqCSV2023}}]\label{def:Ych} Define the map $\Pch:R^{\boldvec{n}} \rightarrow \bigotimes_{i=1}^t\Lambda^{n_i}(X^i)$ 
as follows.  Let $f_{n_i}\in R^{n_i}$ and define 
\[\Pch\left(\bigotimes_{i=1}^t f_{n_i}\right):= \prod_{i=1}^t\ch(f_{n_i}) (X^i),\]
where $\ch$ denotes the ordinary Frobenius characteristic map on $R$ as in~\eqref{eqn:ordinary-charmap-s}. This can be extended multilinearly to all of $R^{\boldvec{n}}$.
In particular for the irreducible character $\chi^{\boldvec{\lambda}}=\bigotimes_{i=1}^t \chi^{\lambda^i}$ indexed by the $t$-tuple ${\boldvec{\lambda}}=(\lambda^1, \ldots, \lambda^t)$, we have 
\[\Pch( \chi^{\boldvec{\lambda}})=\prod_{i=1}^t s_{\lambda^i}(X^i),\]
a product of Schur functions in $t$ different sets of variables.

Expanding in terms of power sum symmetric functions, we obtain, in analogy with \eqref{eqn:ordinary-charmap-p}, for an arbitrary character $\chi$ of $\sym_{\boldvec{n}}$, the formula 
\[\Pch(\chi)=\sum_{\boldvec{\mu}} \chi(\boldvec{\mu}) \prod_{i=1}^t z_{\mu^i}^{-1} \prod_{i=1}^t p_{\mu^i}(X^i),\]
where we have written $\chi(\boldvec{\mu}) $ for the value of the character $\chi$ on the conjugacy class of $\sym_{\boldvec{n}}$ indexed by the $t$-tuple $\boldvec{\mu}=(\mu^1, \ldots, \mu^t)$, $\mu^i\vdash n_i$, and $z_\mu$ is the order of the centraliser in $\sym_n$ of an element of cycle-type $\mu\vdash n$.
\end{dfn}

When $t=1$, $\Pch(\chi)=\ch(\chi)$ for all characters $\chi$ of $\sym_n$, and the product Frobenius characteristic coincides with the ordinary characteristic map.

There is an inner product on $\bigotimes_{i=1}^t\Lambda^{n_i}(X^i)$ defined by 
\begin{equation}\label{eqn:inner-product}
\big\langle \prod_{i=1}^t f_i, \prod_{i=1}^t g_i \big\rangle:=\prod_{i=1}^t  \langle  f_i,  g_i \rangle_{\Lambda^{n_i}(X^i)},
\end{equation}
where $\langle f_i ,g_i\rangle_{\Lambda^{n_i}(X^i)}$ is the usual inner product \cite[Chapter I, Section 7]{Macd1995} in 
the ring of homogeneous symmetric functions $\Lambda^{n_i}(X^i)$ in a single set of variables $X^i$, corresponding to the  inner product of class functions of the symmetric group.

\begin{ex} Let $t=2$ and consider the regular representation $\psi$ of $\sym_2\times \sym_3$.  Then $\psi$  decomposes into irreducibles as follows:
\[\chi^{((2), (3))} +\chi^{((1^2), (3))} +2\chi^{((2), (2,1))} 
+2\chi^{((1^2), (2,1))} + \chi^{((2), (1^3))} +\chi^{((1^2), (1^3))}.\]
Using $X^1$ and $X^2$ for the two sets of variables, we have  
\begin{equation*}
\begin{split}
\Pch(\psi)
&=s_{(2)}(X^1)s_{(3)}(X^2) +s_{(1^2)}(X^1) s_{(3)}(X^2) +2 
s_{(2)}(X^1) s_{(2,1)}(X^2) +2 s_{(1^2)}(X^1) s_{(2,1)}(X^2)\\
&+ s_{(2)}(X^1) s_{(1^3)}(X^2) +s_{(1^2)}(X^1) s_{(1^3)}(X^2)\\
&=h_1^2(X^1) h_1^3(X^2)
\end{split}
\end{equation*}
\end{ex}

We want $\Pch$ to be a ring homomorphism with respect to an induction product akin to~\eqref{eqn:ord-induction}.  In \cite[Definition~3.6]{YLiqCSV2023}, this induction product 
was defined to take an ordered  pair $(\psi, \phi)$ where $\psi$ is a character of $\sym_k\times \sym_\ell$ and $\phi$ is a character of $\sym_m\times \sym_n$, and produce a character of $\sym_{k+m}\times \sym_{\ell+n}$.  For the $t$-fold products, we wish to take a character $\psi$ of  $\sym_{\boldvec{m}}=\bigtimes_{i=1}^t \sym_{m_i}$ and a character $\phi$ of 
$\sym_{\boldvec{n}}=\bigtimes_{i=1}^t \sym_{n_i}$, and map the pair $(\psi, \phi)$
to a character of $\sym_{\boldvec{m}+\boldvec{n}}=\bigtimes_{i=1}^t \sym_{m_i+n_i}$. Here $\boldvec{m}+\boldvec{n}=(m_1+n_1,\ldots, m_t+n_t)$. To do this rigorously, we  first take $\psi$ and $\phi$  to be irreducible characters; the definition will then extend  multilinearly in the obvious way.

\begin{dfn}\label{def:t-fold-IndProduct} Let 
${\boldvec{\lambda}}=(\lambda^1, \ldots , \lambda^t), \lambda^i\vdash m_i$ 
and ${\boldvec{\mu}}=(\mu^1, \ldots , \mu^t), \mu^i\vdash n_i$, so that 
$\chi^{\boldvec{\lambda}}=\bigotimes_{i=1}^t \chi^{\lambda^i}$ and 
$ \chi^{\boldvec{\mu}}=\bigotimes_{i=1}^t \chi^{\mu^i}    $ are respectively irreducible characters of $\sym_{\boldvec{m}}$ and $\sym_{\boldvec{n}}$. 
The $t$-fold \emph{induction product} $\chi^{\boldvec{\lambda}}\circ \chi^{\boldvec{\mu}}$ is then defined to be the induced character 
\begin{equation}\label{eqn:t-fold-ind}
\chi^{\boldvec{\lambda}}\circ \chi^{\boldvec{\mu}}
:=\bigotimes_{i=1}^t (\chi^{\lambda^i}\otimes \chi^{\mu^i})\uparrow_{\sym_{m_i}\times\sym_{n_i}}^{\sym_{m_i+n_i}},
\end{equation}
a character of the direct product $\sym_{\boldvec{m}+\boldvec{n}}$.
Note that this defines a bilinear multiplication 
\[R^{\boldvec{m}}\times R^{\boldvec{n}} \rightarrow R^{\boldvec{m}+\boldvec{n}},\]
for all pairs of $t$-tuples $\boldvec{m}, \boldvec{n}
    \in \bZ_{\ge 0}^t,$ endowing  $\underline{R}=\oplus_{\boldvec{n}\in \bZ_{\ge 0}^t} R^{\boldvec{n}}$
    with the structure of a commutative and associative graded ring with unity, in exact analogy with \cite[p. 112]{Macd1995}.   The grading is now by $t$-tuples $\boldvec{n}$.
    
    We now extend this definition  multilinearly to any pair of representations $\psi$ of $\sym_{\boldvec{m}}=\bigtimes_{i}\sym_{m_i}$ and $\phi$ of $\sym_{\boldvec{n}}=\bigtimes_i \sym_{n_i}$, to produce a new representation $\psi \circ \phi$ of $\sym_{\boldvec{m}+\boldvec{n}}$.
    Explicitly, if $\psi=\sum_{\boldvec{\lambda}}a_{\boldvec{\lambda}}(\psi) \chi^{\boldvec{\lambda}}$ and 
    $\phi=\sum_{\boldvec{\mu}}a_{\boldvec{\mu}}(\phi) \chi^{\boldvec{\mu}}$, then we define the induction product of $\psi$ and $\phi$ to be the following representation of $\sym_{\boldvec{m}+\boldvec{n}}$:
    \begin{equation}\label{eqn:gen-ind-product}\psi \circ \phi:=\sum_{\boldvec{\lambda}, \boldvec{\mu}} a_{\boldvec{\lambda}}(\psi)\,a_{\boldvec{\mu}}(\phi)\  
    (\chi^{\boldvec{\lambda}}\circ \chi^{\boldvec{\mu}}).\end{equation}
\end{dfn}

We note that the dimension of $\chi^{\boldvec{\lambda}}\circ \chi^{\boldvec{\mu}}$ is $\prod_{i=1}^t f^{\lambda^i} f^{\mu^i} \binom{m_i+n_i}{m_i}$, where $f^\lambda$ is the dimension of the irreducible representation of the symmetric group $\sym_{|\lambda|}$ indexed by $\lambda$, i.e. the number of standard Young tableaux of shape $\lambda$.  Hence for the dimension of $\psi\circ \phi$ we have 
\begin{equation}\label{eqn:dim-in-prod}
\dim( \psi\circ \phi)=\dim(\psi)\dim(\phi) \prod_i \binom{m_i+n_i}{m_i}.
\end{equation}
\begin{rem} One could also form the ordinary induced representations $\psi\uparrow_{\sym_{\boldvec{m}}}^{\sym_{m_1+\ldots+m_t}}$  of $\sym_{m_1+\ldots+m_t}$ and  $\phi\uparrow_{\sym_{\boldvec{n}}}^{\sym_{n_1+\ldots+n_t}}$ of $\sym_{n_1+\ldots+n_t}$ 
as well as the induced representation 
$\psi\otimes \phi$ 
from
 $\sym_{\boldvec{m}}\times \sym_{\boldvec{n}}$
to $\sym_{\sum m_i+\sum n_j}$.  However, 
it is the above definition that proves useful in studying the Segre product of Boolean lattices.
\end{rem}

\begin{prop}\label{prop:HomomorphismPch} The map $\Pch$ is a bijective ring homomorphism, with respect to the induction product $\circ$ in $\underline{R}$, from $\underline{R}$ to $\bigotimes_{i=1}^t \Lambda(X^i)$.
Explicitly, if $\boldvec{m}, \boldvec{n} \in \bZ^t_{\ge 0}$ and $\psi$ and $\phi$ are characters of $\sym_{\boldvec{m}}$ and $\sym_{\boldvec{n}}$ respectively, then 
\[\Pch(\psi \circ \phi)=\Pch(\psi)\cdot \Pch(\phi).\]
    \end{prop}

\begin{proof}
It is clear that $\Pch$ is an isomorphism of vector spaces, since by Definition~\ref{def:Ych} it maps basis elements to basis elements.
First we establish that 
\[\Pch(\chi^{\boldvec{\lambda}}\circ \chi^{\boldvec{\mu}} )=\Pch(\chi^{\boldvec{\lambda}}) \Pch( \chi^{\boldvec{\mu}} )\]
for irreducible characters $\chi^{\boldvec{\lambda}}$ of $\sym_{\boldvec{m}}$ and 
$\chi^{\boldvec{\mu}}$ of $\sym_{\boldvec{n}}$. 
We have 
\begin{equation*}
\begin{split} \Pch(\chi^{\boldvec{\lambda}}\circ \chi^{\boldvec{\mu}} )
&= \Pch\left(\bigotimes_{i=1}^t (\chi^{\lambda^i}\otimes \chi^{\mu^i})\uparrow_{\sym_{m_i}\times\sym_{n_i}}^{\sym_{m_i+n_i}}\right)\ \text{ from Definition~\ref{def:t-fold-IndProduct}}\\
&=\prod_{i=1}^t \ch\left((\chi^{\lambda^i}\otimes \chi^{\mu^i})\uparrow_{\sym_{m_i}\times\sym_{n_i}}^{\sym_{m_i+n_i}}\right)\ \text{ from Definition~\ref{def:Ych}}\\
&=\prod_{i=1}^t s_{\lambda^i}(X^i)s_{\mu^i} (X^i), 
\text{ since $\ch$ is a ring homomorphism on $R$ for}\\
&\text{the ordinary induction of characters}\\
&=\prod_{i=1}^t s_{\lambda^i}(X^i)\cdot \prod_{i=1}^t s_{\mu^i}(X^i)\\
&=\Pch(\bigotimes_{i=1}^t \chi^{\lambda^i}) \cdot 
\Pch(\bigotimes_{i=1}^t \chi^{\mu^i})\ \text{again using Definition~\ref{def:Ych}}\\
&=\Pch(\chi^{\boldvec{\lambda}})\cdot \Pch(\chi^{\boldvec{\mu}} ).
\end{split}
\end{equation*} 
Now let $\psi$ and $\phi$ be arbitrary characters of $\sym_{\boldvec{m}}$ and $\sym_{\boldvec{n}}$ respectively, with irreducible decompositions $\psi=\sum_{\boldvec{\lambda}}a_{\boldvec{\lambda}}(\psi) \chi^{\boldvec{\lambda}}$ and 
    $\phi=\sum_{\boldvec{\mu}}a_{\boldvec{\mu}}(\phi) \chi^{\boldvec{\mu}}$.
From~\eqref{eqn:gen-ind-product} and the linearity of the map $\Pch$ we have 
\begin{align*}\Pch(\psi\circ\phi)&=\sum_{\boldvec{\lambda}, \boldvec{\mu}} a_{\boldvec{\lambda}}(\psi)\,a_{\boldvec{\mu}}(\phi)\  
    \Pch(\chi^{\boldvec{\lambda}}\circ \chi^{\boldvec{\mu}})\\
    &=\sum_{\boldvec{\lambda}, \boldvec{\mu}} a_{\boldvec{\lambda}}(\psi)\,a_{\boldvec{\mu}}(\phi)
    \Pch(\chi^{\boldvec{\lambda}}) \Pch (\chi^{\boldvec{\mu}})\\
    &=\sum_{\boldvec{\lambda}}a_{\boldvec{\lambda}}(\psi) \Pch(\chi^{\boldvec{\lambda}})\cdot 
    \sum_{\boldvec{\mu}}a_{\boldvec{\mu}}(\phi) \Pch(\chi^{\boldvec{\mu}})\\
    &=\Pch(\psi)\cdot \Pch(\phi),
    \end{align*}
    which finishes the proof.
\end{proof}

The next proposition shows the equivalence of the above definition of induction product with the one in \cite{YLiqCSV2023}.  Before proving this proposition, some explanation is in order.  It is clear that $\bigtimes_{i=1}^t (\sym_{m_i}\times \sym_{n_i})$ is a direct product of  Young subgroups in $\sym_{\boldvec{m}+\boldvec{n} }$. The key point here is that we can  also view  $\sym_{\boldvec{m}}\times  \sym_{\boldvec{n}} $ as a subgroup of $ \sym_{\boldvec{m}+\boldvec{n} }$. In order to do this, we view all the permutations as acting on $2t$ disjoint sets of symbols $(\sqcup_{i=1}^t M_i)\sqcup (\sqcup_{j=1}^t N_j)$, where $|M_i|=m_i$ and $|N_i|=n_i$, $1\le i\le t.$  (Here $\sqcup$ denotes disjoint union.) 

Let $\sigma_i\in \sym_{m_i}=\sym (M_i), \tau_i\in \sym_{n_i}=\sym (N_i)$, $i=1, \ldots, t$, where $\sym (M_i)$ is the set of permutations on the letters in the set $M_i$, etc.
Thus the $t$-tuple $(\sigma_1, \ldots, \sigma_t)\in \sym_{\boldvec{m}}$ may be viewed as a product of commuting permutations $\prod_{i=1}^t \sigma_i$, and similarly for a $t$-tuple $(\tau_1, \ldots, \tau_t)\in \sym_{\boldvec{n}}$.  Likewise, 
the $t$-tuple of ordered pairs $((\sigma_1, \tau_1), \ldots, (\sigma_t, \tau_t))$ in $\bigtimes_{i=1}^t (\sym_{m_i}\times \sym_{n_i})$ can be identified with the product 
of commuting permutations $\prod_{i=1}^t \sigma_i \tau_i$.  With this identification, 
 $ \sym_{\boldvec{m}}\times  \sym_{\boldvec{n}} $  and $\bigtimes_{i=1}^t (\sym_{m_i}\times \sym_{n_i})$ coincide as subgroups of $ \sym_{\boldvec{m}+\boldvec{n} }$, which is identified with the direct product 
$\bigtimes_{i=1}^t \sym(M_i\sqcup N_i)$. 
\begin{prop}\label{prop:t-fold-ind-prod2} Let $\chi^{\boldvec{\lambda}}=\bigotimes_{i=1}^t \chi^{\lambda^i}$ and $\chi^{\boldvec{\mu}}=\bigotimes_{i=1}^t \chi^{\mu^i}$ be irreducible characters of $\sym_{\boldvec{m}}$ and $\sym_{\boldvec{n}}$, as in Definition~\ref{def:t-fold-IndProduct}. 
Then 
\begin{equation}\label{eqn:ind-prod-equiv} 
\chi^{\boldvec{\lambda}}\circ \chi^{\boldvec{\mu}}=
\left(\chi^{\boldvec{\lambda}}\otimes \chi^{\boldvec{\mu}}   \right)\big\uparrow_{(\bigtimes_{i=1}^t \sym_{m_i})\times  (\bigtimes_{i=1}^t \sym_{n_i})}^{\bigtimes_{i=1}^t \sym_{m_i+n_i}  },
\end{equation}
which can be rewritten 
\begin{equation}\label{eqn:ind-prod-equiv2}
\chi^{\boldvec{\lambda}}\circ \chi^{\boldvec{\mu}}
=\left(\chi^{\boldvec{\lambda}}\otimes \chi^{\boldvec{\mu}}   \right)\big\uparrow_{\sym_{\boldvec{m}}\times  \sym_{\boldvec{n}}}^{\sym_{\boldvec{m}+\boldvec{n} } }.
\end{equation}
More generally, if $\psi$ is a character of $\sym_{\boldvec{m}}$ and $\phi$ is a character of $\sym_{\boldvec{n}}$, then 
\begin{equation}\label{eqn:ind-prod-equiv-gen}
\psi\circ \phi = (\psi\otimes \phi)\big\uparrow_{\sym_{\boldvec{m}}\times  \sym_{\boldvec{n}}}^{\sym_{\boldvec{m}+\boldvec{n} } }.
\end{equation}
\end{prop}

\begin{proof}
A straightforward  computation (e.g., using the definition of induced module via cosets) shows that if $V_i$ is a representation of a  subgroup $H_i$ of a group $G_i$, $1\le i\le t$, then 
\begin{equation}\label{eqn:ind-distribute} \bigotimes_{i=1}^t (V_i\big\uparrow_{H_i}^{G_i}) \simeq (\bigotimes_{i=1}^t V_i)\big\uparrow_{\times_{i=1}^t H_i}^{\times_{i=1}^t G_i}.
\end{equation}

We now have 
\begin{align*} \chi^{\boldvec{\lambda}}\circ \chi^{\boldvec{\mu}}& = 
\bigotimes_{i=1}^t (\chi^{\lambda^i}\otimes \chi^{\mu^i})\big\uparrow_{\sym_{m_i}\times\sym_{n_i}}^{\sym_{m_i+n_i}}\\
&=\left(\bigotimes_{i=1}^t (\chi^{\lambda^i}\otimes \chi^{\mu^i})\right)\big\uparrow_{\bigtimes_{i=1}^t(\sym_{m_i}\times\sym_{n_i})}^{\bigtimes_{i=1}^t\sym_{m_i+n_i}} \text{ by ~\eqref{eqn:ind-distribute}}.
\end{align*}
Now we  claim that 
\begin{equation}\label{eqn:key-reordering}
\left(\bigotimes_{i=1}^t (\chi^{\lambda^i}\otimes\chi^{\mu^i}) \right)
\big\uparrow_{\bigtimes_{i=1}^t(\sym_{m_i}\times\sym_{n_i})}^{\sym_{\boldvec{m}+\boldvec{n}}}
=(\chi^{\boldvec{\lambda}}\otimes \chi^{\boldvec{\mu}})\big\uparrow_{ \sym_{\boldvec{m}}\times  \sym_{\boldvec{n}}}^{\sym_{\boldvec{m}+\boldvec{n}}}.
\end{equation}
 Write  $f$ for the character $\bigotimes_{i=1}^t (\chi^{\lambda^i}\otimes \chi^{\mu^i})$ of $\bigtimes_{i=1}^t (\sym_{m_i}\times \sym_{n_i})$ and $g$ for the character $\chi^{\boldvec{\lambda}}\otimes \chi^{\boldvec{\mu}}=(\bigotimes_{i=1}^t \chi^{\lambda^i}) \otimes (\bigotimes_{i=1}^t \chi^{\mu^i}) $ of $\sym_{\boldvec{m}}\times \sym_{\boldvec{n}}$. In view of the identifications made immediately preceding Proposition \ref{prop:t-fold-ind-prod2},  we have, for $\sigma_i\in \sym_{m_i}, \tau_i\in \sym_{n_i}$, $1\le i \le t$, 
\begin{equation*}
\begin{split}
&f((\sigma_1, \tau_1), \ldots, (\sigma_t, \tau_t)  )=\prod_{i=1}^t (\chi^{\lambda^i}\otimes \chi^{\mu^i})(\sigma_i, \tau_i)= \prod_{i=1}^t  \chi^{\lambda^i}(\sigma_i) \chi^{\mu^i}(\tau_i),\\
&g((\sigma_1, \ldots, \sigma_t) ,  (\tau_1, \ldots, \tau_t))
=\chi^{\boldvec{\lambda}}(\sigma_1, \ldots, \sigma_t)\cdot \chi^{\boldvec{\mu}}(\tau_1, \ldots, \tau_t)
= \prod_{i=1}^t  \chi^{\lambda^i}(\sigma_i)  \prod_{i=1}^t\chi^{\mu^i}(\tau_i),
\end{split}
\end{equation*}
which shows that the two expressions coincide, establishing~\eqref{eqn:ind-prod-equiv}.  Finally  \eqref{eqn:ind-prod-equiv-gen} follows by multilinearity.
\end{proof}

We illustrate these ideas with the case $t=2$. Let $\psi$ be a character of $\sym_{m_1}\times \sym_{m_2}$, and $\phi$ a character of $\sym_{n_1}\times \sym_{n_2}$.   Then by definition of the induction product and Proposition~\ref{prop:t-fold-ind-prod2}, $\psi\circ \phi$ is the induced module 
$(\psi\otimes \phi)\big\uparrow_{(\sym_{m_1}\times \sym_{m_2} )\times (\sym_{n_1}\times \sym_{n_2} )}^{\sym_{m_1+n_1}\times \sym_{m_2+n_2}}.$  Again, the key point established in the proof of~\eqref{eqn:key-reordering} above is that $\psi\otimes \phi $ can be viewed as a character of $(\sym_{m_1}\times \sym_{n_1} )\times (\sym_{m_2}\times \sym_{n_2} )$.

We will use the following important special case in the next section.
\begin{cor}\label{cor:all-parts-equal} Let $\psi$ be a character of the $t$-fold direct product $\sym_r^{\times t}$ and let $\phi$ be a character of the $t$-fold direct product $\sym_{n-r}^{\times t}$.
Then \[(\psi\otimes \phi)\uparrow_{\sym_r^{\times t}\times \sym_{n-r}^{\times t}}^{\sym_n^{\times t}}= \psi \circ \phi, \text{ and hence }
\Pch\left((\psi\otimes \phi)\uparrow_{\sym_r^{\times t}\times \sym_{n-r}^{\times t}}^{\sym_n^{\times t}}\right)= \Pch(\psi) \Pch(\phi).
\]
\end{cor}

\section{The actions of $\sym_n^{\times t}$ and $\sym_n$ on the homology of $B_n^{(t)}$}\label{sec:tfold-Boolean-lattice-repn}

In this section we begin by generalizing \cite[Theorem~4.1]{YLiqCSV2023} to the $t$-fold Segre power of Boolean lattices $B_n^{(t)}$ for any $t$.  

From work of Bj\"orner and Welker \cite[Theorem 1, Corollary 9]{Segre_rees} it is known that the Segre product preserves the property of being (homotopy) Cohen-Macaulay, and hence by iteration, from the well-known fact that the Boolean lattice $B_n$ is homotopy Cohen-Macaulay, we know that the same is true for $B_n^{(t)}$.  In particular every open interval $(x,y)=((x_1,\ldots, x_t),(y_1,\ldots, y_t))$, where $0\le |x_i|=r<s=|y_j|\le n$ for all $1\le i,j\le t$, has the homotopy type of a wedge of spheres in the top dimension $s-r-2$, and hence the homology of $(x,y)$ vanishes in all but this top dimension.

We will use the Whitney homology technique of \cite{SundaramAIM1994} to determine the action of $\sym_n^{\times t}$ on the top homology module $\tilde{H}_{n-2}(B_n^{(t)})$. 

\begin{thm}[{\cite[Lemma 1.1 and Theorem 1.2]{SundaramAIM1994}}]\label{thm:Whit}
Let $Q$ be a bounded and ranked Cohen-Macaulay poset of rank $n$, and let $G$ be a finite group of automorphisms of $Q$.   Let $W\!H_r(Q)$ denote its $r$th Whitney  homology, defined by \[W\!H_r(Q)=\bigoplus_{x\in Q,\, \rk(x)=r} \tilde{H}_{r-2}(\hat 0, x).\] Then $W\!H_r(Q)$ is a $G$-module, and as virtual $G$-modules one has the identity 
\[\tilde{H}_{n-2}(Q)=\sum_{r=0}^{n-1} (-1)^{n-r+1} W\!H_r(Q).\]
\end{thm}

\begin{thm}\label{thm:def-rec-tfoldSegreBn} Fix $t\ge 1$. Set $\beta_0^{(t)}=1$, and for $n\ge 1$ denote by $\beta_n^{(t)}$ the product Frobenius characteristic $\Pch(\tilde{H}_{n-2}(B_n^{(t)}))$ of the top homology of the $t$-fold Segre power $B_n^{(t)}$.  Then $\beta_n^{(t)}$  satisfies the recurrence 
\begin{equation*} 
\sum_{i=0}^n (-1)^i \beta_i^{(t)}\prod_{j=1}^t h_{n-i}(X^j) =0.
\end{equation*}
\end{thm}

\begin{proof} When $t=1$, $B^{(t)}_n=B_n$ and $\beta^{(1)}_n=\beta_n$ is simply the Frobenius characteristic of the top homology of the Boolean lattice, so $\beta^{(1)}_n=\beta_n=e_n.$
In this case the above equation reduces to the well-known symmetric function identity 
$e_n=\sum_{i=0}^{n-1} (-1)^{n-i-1}e_i h_{n-i}$.

    We apply Theorem~\ref{thm:Whit} to the  Cohen-Macaulay poset $Q=B_n^{(t)}$, which has rank $n$. We need to compute the $\sym_n^{\times t}$-module structure of the Whitney homology $W\!H_r(Q)$. Let $x_0$ be the $t$-tuple $([r], [r], \ldots, [r])$, where $[r]=\{1,2,\ldots, r\}$; $x_0$ has rank $r$ in $Q$.   The stabiliser of  $x_0$ at rank $r$, and hence of the interval $(\hat 0, x_0)$,  is  $(\sym_r \times \sym_{n-r})^{\times t}$. The orbit of $x_0$ under the action of $\sym_n^{\times t}$  generates  all other elements at rank $r$,  and hence  $W\!H_r(Q)$ is $\sym_n^{\times t}$-isomorphic to a  module induced from the direct product  ${(\sym_r \times \sym_{n-r})^{\times t}}$.
    
    The copies of $\sym_{n-r}$ act trivially on $x_0$. More precisely, if  $((\sigma_1, \tau_1), \ldots, (\sigma_t, \tau_t))$ is an element of $(\sym_r \times \sym_{n-r})^{\times t}$, then the $t$-tuple $(\sigma_1,\ldots, \sigma_t)$ acts like the representation $\tilde{H}_{r-2}(B_r^{(t)})$, with product Frobenius characteristic  $\beta_r^{(t)}$, and the $t$-tuple $(\tau_1\ldots, \tau_t)$ acts trivially, i.e., like the representation $\bigotimes_{i=1}^t \mathbbm{1}_{\sym_{n-r}}$.  
    
    This is 
     the induction product of Definition~\ref{def:t-fold-IndProduct}, and by Proposition~\ref{prop:t-fold-ind-prod2} and Corollary~\ref{cor:all-parts-equal}, the induced module $W\!H_{r-2}(Q)$   coincides with 
    $\tilde{H}_{r-2}({B_r^{(t)}}) \circ (\mathbbm{1}_{\sym_{n-r}})^{\otimes t}$, whose product Frobenius characteristic is precisely $\beta_r^{(t)} \prod_{j=1}^t h_{n-r}(X^j)$.
\end{proof}
The dimension of the $r$th Whitney homology is given by 
\[\dim\left(\tilde{H}_{r-2}({B_r^{(t)}})\right) \binom{n}{r}^t=(-1)^r\mu({B_r^{(t)}})\binom{n}{r}^t=w_r^{(t)}\binom{n}{r}^t, \]
 where  $w^{(t)}_r$  is the number defined in Proposition~\ref{prop:mu-Bnt}. Hence Theorem~\ref{thm:def-rec-tfoldSegreBn} is the group-equivariant version of~\eqref{eqn:dim-Segre-homology}.
\begin{cor}\label{cor:mult} Let $n\ge 1$. In the $\sym_n^{\times t}$-module $\tilde{H}_{n-2}(B_n^{(t)})$, the multiplicity of 
\begin{enumerate}
    \item
the trivial representation $\bigotimes_{j=1}^t \chi^{(n)}$ is 
zero unless $n=1$, in which case it is 1; 
\item the sign representation $\bigotimes_{j=1}^t \chi^{(1^n)}$ is 1 for all $n\ge 1$.
\end{enumerate}
\end{cor}
\begin{proof}  We deduce this from the recurrence of Theorem~\ref{thm:def-rec-tfoldSegreBn}, which we rewrite as 
\begin{equation}\label{eqn:rec-temp}
  \beta_n^{(t)}=\sum_{i=0}^{n-1} (-1)^{n-1-i}\beta_i^{(t)}\prod_{j=1}^t h_{n-i}(X^j). 
\end{equation} 
Now $\prod_{j=1}^t h_n(X^j)$ and $\prod_{j=1}^t e_n(X^j)$ are respectively the product Frobenius characteristics of the trivial and sign representations of $\sym_n^{\times t}$.  We use the inner product~\eqref{eqn:inner-product} to compute the multiplicities. 
We have, for all $n$,  using  the Kronecker delta $\delta_{n,1}$ which equals 1 if $n=1$ and zero otherwise:
\begin{align*}
&\langle \prod_{j=1}^t h_n(X^j) , \prod_{j=1}^t h_n(X^j)\rangle =\prod_{j=1}^t\langle h_n(X^j) , h_n(X^j) \rangle=1,\\
&\langle \prod_{j=1}^t h_n(X^j) , \prod_{j=1}^t e_n(X^j)\rangle =\prod_{j=1}^t\langle h_n(X^j) , e_n(X^j) \rangle=\delta_{n,1}.
\end{align*}
We compute 
\begin{align*}
\langle \beta_r^{(t)} \prod_{j=1}^t h_{n-r}(X^j), \prod_{j=1}^t h_{n}(X^j)\rangle 
&= \langle \beta_r^{(t)} \prod_{j=1}^t h_{n-r}(X^j), \prod_{j=1}^t h_{r}(X^j) \prod_{j=1}^t h_{n-r}(X^j)\rangle\\
&=\langle \beta_r^{(t)} , \prod_{j=1}^t h_{r}(X^j)\rangle \cdot \langle \prod_{j=1}^t h_{n-r}(X^j), \prod_{j=1}^t h_{n-r}(X^j)\rangle\\
&=\langle \beta_r^{(t)} , \prod_{j=1}^t h_{r}(X^j)\rangle\, \prod_{j=1}^t \langle h_{n-r}(X^j), h_{n-r}(X^j)\rangle.
\end{align*}

Since 
$\langle \beta_r^{(t)} , \prod_{j=1}^t h_{r}(X^j)\rangle=1$ for $r=0,1$, by induction  the recurrence~\eqref{eqn:rec-temp} gives Item (1).

Similarly, for $0\le r\le n-1$, we have 
\begin{align*}
\langle \beta_r^{(t)} \prod_{j=1}^t h_{n-r}(X^j), \prod_{j=1}^t e_{n}(X^j)\rangle 
&=\langle \beta_r^{(t)} , \prod_{j=1}^t e_{r}(X^j)\rangle  \delta_{n-r, 1}.
\end{align*}
Hence the recurrence~\eqref{eqn:rec-temp} now gives, for $n\ge 3$,
$\langle \beta_n^{(t)} , \prod_{j=1}^t e_{n}(X^j)\rangle =\langle \beta_{n-1}^{(t)} , \prod_{j=1}^t e_{n-1}(X^j)\rangle$.

By direct calculation, the multiplicity of the sign representation is zero in $\beta_n^{(t)}$ when $n=0,1$, and it is 1 for $n=2$. 
Item (2) now follows from~\eqref{eqn:rec-temp} by  induction.
 \end{proof}

Note the agreement with the case $t=1$, when $B_n^{(t)}$ is simply the Boolean lattice $B_n$, and so its homology carries the sign representation of $\sym_n$.

\begin{ex}\label{ex:small-reps} Let $t=2$. We use the recurrence to compute some of the symmetric functions $\beta^{(2)}_n$ in two sets of variables $X^1$, $X^2$. We have  
$\beta^{(2)}_0=1$ and 
$\beta^{(2)}_1=h_1(X^1) h_1(X^2),$ the product characteristic of the trivial representation of $\sym_1\times \sym_1$. Then the recurrence gives 
$\beta^{(2)}_2=\beta^{(2)}_1(h_1(X^1) h_1(X^2))-h_2(X^1) h_2(X^2)
=h_1^2(X^1) h_1^2(X^2) -h_2(X^1) h_2(X^2)$, and so 
\[\beta^{(2)}_2=e_2(X^1) h_2(X^2)+h_2(X^1)e_2(X^2)+e_2(X^1) e_2(X^2).\]
Similarly,
\begin{align*} 
\beta^{(2)}_3&=\beta^{(2)}_2 h_1(X^1) h_1(X^2) -\beta^{(2)}_1 h_2(X^1) h_2(X^2) +\beta^{(2)}_0 h_3(X^1) h_3(X^2)\\
&=h_1^3(X^1)h_1^3(X^2)-2 h_2(X^1) h_2(X^2) h_1(X^1) h_1(X^2)+
    h_3(X^1) h_3(X^2)\\
    &=(h_3(X^1) e_3(X^2) + e_3(X^1) h_3(X^2)) + e_3(X^1) e_3(X^2) 
    +s_{(2,1)}(X^1) s_{(2,1)}(X^2)\\
    &+2\left(s_{(2,1)}(X^1)e_3(X^2)+e_3(X^1)s_{(2,1)}(X^2)\right) .
\end{align*}
\end{ex}

By definition of the product Frobenius characteristic and the fact that $\tilde{H}_{n-1}(B_n^{(2)})$ is a true $(\sym_n\times \sym_n)$-module, it follows that $\beta^{(2)}_n$ must have a positive expansion in the basis $\{s_\lambda(X^1) s_\mu(X^2): \lambda, \mu\vdash n\}$. This is confirmed by the above examples.

\begin{dfn}\label{def:Zn and Zlambda}
Define $Z^{(t)}_{i}:=\prod_{j=1}^t h_i(X^j)$, and define the degree of $Z^{(t)}_i$ to be $i$. 

Also define, for each 
$\lambda\vdash n$, 
$Z^{(t)}_\lambda=\prod_j Z^{(t)}_{\lambda_j}$. Thus $Z^{(t)}_\lambda=\prod_{j=1}^t h_{\lambda}(X^j)$.
\end{dfn}

\begin{lem}
  The symmetric function $\beta^{(t)}_n$ can be written as a polynomial of homogeneous degree $n$ in $\{Z^{(t)}_i: 1\le i\le n\}$ such that for $n\ge 2$, the sum of the coefficients is 0. 
 \end{lem}
 \begin{proof}
The $\beta^{(t)}_n$ are defined recursively by 
$\beta^{(t)}_1=Z_1^{(t)}$, $\beta^{(t)}_0=1$ and 
\begin{equation}\label{eqn:rec-beta}
\beta^{(t)}_n=\sum_{i=0}^{n-1} (-1)^{n-1-i} \beta^{(t)}_i Z_{n-i}^{(t)}.
\end{equation}
  The statement follows by induction using Corollary~\ref{cor:mult}, since each monomial in $Z^{(t)}_i$ contributes 1 to the multiplicity of the trivial representation, while the contribution from $\beta^{(t)}_n$ is 0 except for $\beta^{(t)}_1$ and $\beta^{(t)}_0$, where the contribution is 1.  \end{proof}

Recall  the recurrence for the Frobenius characteristic of the 
top homology of the Boolean lattice, namely, the symmetric function identity in $\Lambda_n(X)$ (in one set of variables $X$) given by \cite{Macd1995}
\begin{equation}\label{eqn:Booleanrec} e_n=\sum_{i=0}^{n-1} (-1)^{n-1-i} e_i h_{n-i} .\end{equation}
 This is equivalent to the generating function identity \cite{Macd1995}
\begin{equation}\label{eqn:Booleangf} \sum_{n\ge 0} u^n e_n=(\sum_{u\ge 0} u^n (-1)^n h_n)^{-1}.\end{equation}

\begin{dfn}\label{def:map-Phi-S}
 Define a map $\Phi_t:\Lambda(X) \rightarrow \bigotimes_{j=1}^t \Lambda(X^j)$ by setting 
\begin{equation}\label{eqn:mapPhi}
\Phi_t(h_n):=\prod_{j=1}^t h_n(X^j)=Z^{(t)}_n,
\end{equation}
and extending multiplicatively and linearly to all of $\Lambda(X)$.
This is well defined since the $h_n$ are algebraic generators for  $\Lambda(X)$.
\end{dfn}

\begin{prop}\label{prop:Phi-en-is-tfold-betaBoolean}  The map $\Phi_t:\Lambda(X)\rightarrow \bigotimes_{j=1}^t \Lambda(X^j)$  is an injective degree-preserving algebra homomorphism 
such that 
\[\Phi_t(e_n)=\beta_n^{(t)}.\] 
Moreover, $\{\beta^{(t)}_n\}_n$ is an algebraically independent set.
\end{prop}
\begin{proof}  Since the $\{h_\lambda(X^j)\}_\lambda$ are linearly independent in $\Lambda_n(X^j)$, the set $\{ Z^{(t)}_\lambda\}_{\lambda\vdash n}$ is linearly independent in $\bigotimes_{j=1}^t \Lambda_n(X^j)$, and so  the $Z^{(t)}_n$ are algebraically independent in $\bigotimes_{j=1}^t \Lambda(X^j)$.
The $h_n,\, n\ge 1$, are algebraically independent generators for the ring $\Lambda(X)$ \cite{Macd1995}.
Thus $\Phi_t$ extends to an injective degree-preserving algebra homomorphism 
$\Lambda(X)\rightarrow \bigotimes_{j=1}^t \Lambda(X^j)$, so that 
$\Phi_t(\Lambda_n(X))\subset \bigotimes_{j=1}^t \Lambda_n(X^j)$.  

In particular $\Phi_t(h_\lambda)=Z_\lambda^{(t)}$. Applying $\Phi_t$ to 
\eqref{eqn:Booleanrec} gives precisely the recurrence ~\eqref{eqn:rec-beta}, with the same initial conditions, since $\Phi_t(e_1)=\Phi_t(h_1)=Z^{(t)}_1$, $\Phi_t(e_0)=1=Z^{(t)}_0$.   Hence 
$\Phi_t(e_n)=\beta_n^{(t)}$.

The injectivity of $\Phi_t$ now implies that the $\beta_n^{(t)}$ must be algebraically independent.
\end{proof}

\begin{ex}\label{ex:Phi-S} To illustrate the workings of the map $\Phi_t$, we compute $\Phi_2(s_{(n-1,1)})$ for  $t=2$.  Since $s_{(n-1,1)}=h_{n-1}h_1-h_n,$ we have, 
for the two sets of variables $X^1, X^2$ as in  Example~\ref{ex:small-reps},
\begin{align*} \Phi_2(s_{(n-1,1)})&=\Phi_2(h_{n-1}h_1)-\Phi_2(h_n)\\
                                                  &=\Phi_2(h_{n-1})\Phi_2(h_1)-\Phi_2(h_n)\\
                                                 &=h_{n-1}(X^1)h_{n-1}(X^2)\, h_1(X^1)h_1(X^2)- h_n(X^1) h_n(X^2)\\
                                                  &=(s_{({n-1},1)}(X^1)+s_{(n)}(X^1))\,(s_{({n-1},1)}(X^2)+s_{(n)}(X^2)) -h_n(X^1) h_n(X^2)\\
                                           &=s_{({n-1},1)}(X^1)s_{({n-1},1)}(X^2)+ (s_{({n-1},1)}(X^1)s_{(n)}(X^2) + s_{({n-1},1)}(X^2) s_{(n)}(X^1)).
\end{align*}
For $t=3$ we would similarly have, 
 for the three sets of variables $X^i$, $i=1,2,3$,
\begin{align*} \Phi_3(s_{(n-1,1)})
                                &=\Phi_3(h_{n-1})\Phi_3(h_1)-\Phi_3(h_n)\\
                                &=h_{n-1}(X^1)h_{n-1}(X^2)h_{n-1}(X^3)\, h_1(X^1)h_1(X^2)h_{1}(X^3)- h_n(X^1) h_n(X^2)h_{n}(X^3)\\
                            &=s_{({n-1},1)}(X^1)s_{({n-1},1)}(X^2) s_{({n-1},1)}(X^3)+ (s_{({n-1},1)}(X^1) s_{(n)}(X^2) s_{({n-1},1)}(X^3)\\
                            &+  s_{({n-1},1)}(X^2) s_{(n)}(X^1) s_{({n-1},1)}(X^3) +s_{(n)}(X^1) s_{(n)}(X^2) s_{(n-1,1)}(X^3)\\
                            &+ s_{({n-1},1)}(X^1)s_{({n-1},1)}(X^2) s_{(n)}(X^3)\\
                            &+ s_{({n-1},1)}(X^1) s_{(n)}(X^2) s_{(n)}(X^3) +  s_{({n-1},1)}(X^2) s_{(n)}(X^1) s_{(n)}(X^3) 
                            .
\end{align*}
\end{ex}

  Example~\ref{ex:Phi-S} is a special case of Item (1)  of Proposition~\ref{prop:Phi-Jacobi-Trudi} below. Using the Jacobi-Trudi expansion \cite[Eqns. (3,4), (3,5)]{Macd1995} of the Schur function $s_\lambda$ in the basis of homogeneous symmetric functions $h_\mu$, from Definition~\ref{def:Zn and Zlambda} we obtain 
%
\begin{prop}\label{prop:Phi-Jacobi-Trudi} For $\lambda\vdash n$, let $\lambda'$ denote the conjugate partition of $\lambda$. Then 
\begin{enumerate}
\item $\Phi_t(s_\lambda)=\det(Z^{(t)}_{\lambda_i-i+j})_{1\le i,j\le \ell(\lambda)}$,
where $Z^{(t)}_0=1$ and we set $Z^{(t)}_m=0$ if $m<0$.
\item
$\Phi_t(s_{\lambda'})=\det(\beta^{(t)}_{\lambda_i-i+j})_{1\le i,j\le \ell(\lambda)}$, where 
$\beta^{(t)}_0=1$ and we set $\beta^{(t)}_m=0$ if $m<0$.
\end{enumerate}
\end{prop}
\begin{proof} The  two items are clear from the definitions.
\end{proof}

 For an integer partition $\lambda\vdash n$ of $n$, we write $\ell(\lambda)$ for the total number of parts of $\lambda$, and $m_i(\lambda)$ for the number of parts equal to $i$.   Also let 
$K_{\mu,\nu}$ be the Kostka number, i.e., the number of semistandard Young tableaux of shape $\mu\vdash n$ and weight $\nu\vdash n$ \cite[Eqn. (5.12)]{Macd1995}.  In particular, 
  $f^\mu=K_{\mu, (1^n)}$ is the number of standard Young tableaux of shape $\mu$.

\begin{dfn}\label{def:c-lambda} For $\lambda\vdash n$ with $m_i(\lambda)$ parts of size $i$ and number of parts $\ell(\lambda)$, define $c_\lambda$ to be the integer 
\[c_\lambda=(-1)^{n-\ell(\lambda)}\frac{\ell(\lambda)!}{\prod_i m_i(\lambda)!}.\]
\end{dfn}

The integers $c_\lambda$  play an important role in the irreducible decomposition of $\beta_n^{(t)}$, as we prove next.  First we record a fact that will be needed in the proof.
\begin{lem}\label{lem:Zlambda-into- Schurs} The multiplicity of the $\sym_n^{\times t}$-irreducible 
indexed by the $t$-tuple of partitions $\boldvec{\mu}=(\mu^1,\ldots,\mu^t)$,  $\mu^j\vdash n$, $1\le j\le t$, in the $\sym_n^{\times t}$-module with product Frobenius characteristic $Z_\lambda^{(t)}$ is \begin{center}$\prod_{j=1}^t K_{\mu^j, \lambda}$.\end{center}
\end{lem}
\begin{proof}  
We use the well-known expansion \cite{Macd1995}  $h_\lambda=\sum_{\mu\vdash n} K_{\mu, \lambda}\, s_\mu$. This gives,  for each set of variables $X^j$, $1\le j\le t$, 
$h_\lambda(X^j)=\sum_{\mu^j\vdash n} K_{\mu^j, \lambda}\,s_{\mu^j}(X^j).$

Since by definition, 
$Z_\lambda^{(t)}=\prod_{j=1}^t h_\lambda(X^j)$, the result follows.
\end{proof}

If $\{u_\lambda\}, 
\{v_\lambda\}$ are two sets of bases for the ring of symmetric functions $\Lambda_n(X)$, as in \cite[Ch1, Sec 6]{Macd1995} we write $\mathcal{M}(u,v)$ for the transition matrix from the basis $\{u_\lambda\}$  
to the basis 
$\{v_\lambda\}$.  More precisely, \begin{equation}\label{eqn:def-transition-matrix}
u_\lambda=\sum_\mu \mathcal{M}(u,v)_{\lambda, \mu}\, v_\mu.
\end{equation}

\begin{thm}\label{thm:Bnt-irreps} 
For the product Frobenius characteristic $\beta_n^{(t)}$ of the top homology of $B_n^{(t)}$, we have:
\begin{enumerate}
\item 
$\beta_n^{(t)}=\sum_{\lambda\vdash n} c_\lambda Z^{(t)}_\lambda$.
\item $\sum_{n\ge 0} u^n\beta^{(t)}_n=(\sum_{u\ge 0} u^n (-1)^n Z^{(t)}_n)^{-1}$.
\item The multiplicity of the $\sym_n^{\times t}$-irreducible 
indexed by the $t$-tuple of partitions $\boldvec{\mu}=(\mu^1,\ldots,\mu^t)$,  $\mu^j\vdash n$, $1\le j\le t$, 
 in $\tilde{H}_{n-2}(B_n^{(t)})$ 
equals 
\[ c^t_{\boldvec{\mu}}=\sum_{\lambda\vdash n} c_\lambda \prod_{j=1}^t K_{\mu^j, \lambda}.\]

\item Let  $\mathcal{M}(s,h)$ denote the transition matrix from the basis of Schur functions to the basis of homogeneous symmetric functions. 

 The multiplicity of the $\sym_n^{\times t}$-irreducible indexed by the $t$-tuple of partitions $(\mu^1, \ldots, \mu^t)$, $\mu^i\vdash n$, $1\le i\le t$, in the (possibly virtual) module whose product Frobenius characteristic is 
\begin{equation}\label{eqn:Phi-t-irreps}
\langle \Phi_t(s_\lambda), \prod_{j=1}^ts_{\mu^j}(X^j)   \rangle=
\sum_{\nu\vdash n} \mathcal{M}(s,h)_{\lambda, \nu} \prod_{j=1}^t K_{\mu^j, \nu}.\end{equation}
\end{enumerate}
\end{thm}
\begin{proof} 
Let $\mathcal{M}(e,h)$ 
denote the transition matrix from the basis of elementary symmetric functions  
to the basis of homogeneous symmetric functions  \cite[Chapter 1, Section 6]{Macd1995}.  Then 
\[e_n=\sum_{\lambda\vdash n} \mathcal{M}(e,h)_{(n), \lambda} h_\lambda,\] 
and hence, applying $\Phi_t$ and using Proposition~\ref{prop:Phi-en-is-tfold-betaBoolean}, 
\[\beta_n^{(t)}=\sum_{\lambda\vdash n} \mathcal{M}(e,h)_{(n), \lambda} Z^{(t)}_\lambda.\] 
Since $\mathcal{M}(s,h)$ is the transition matrix from the basis of Schur functions to the basis of homogeneous symmetric functions, its 
inverse $\mathcal{M}(h,s) $  is the transpose of the Kostka matrix $K=(K_{\lambda, \mu})$, and 
$\mathcal{M}(e,h)_{(n), \lambda}
=\mathcal{M}(s,h)_{(1^n), \lambda}=(K^{-1})_{\lambda, (1^n)}.$
E{\v g}ecio{\v g}lu and Remmel \cite[Corollary 1, Part (iv)]{EgeciogluRemmel1990} give an explicit formula for  the signed numbers $(K^{-1})_{\lambda, (1^n)}$, namely, $(K^{-1})_{\lambda, (1^n)}=(-1)^{n-\ell(\lambda)}\frac{\ell(\lambda)!}{\prod_i m_i(\lambda)!}$. These are precisely the numbers $c_\lambda$ in Definition~\ref{def:c-lambda}.
This establishes  (1).

Similarly, applying $\Phi_t$ to \eqref{eqn:Booleangf} establishes  (2). 

Item (3) is now immediate from (1) by applying Lemma~\ref{lem:Zlambda-into- Schurs}.
 
Item (4)  results from applying $\Phi_t$ to the identity 
\[s_\lambda=\sum_{\nu\vdash n} \mathcal{M}(s,h)_{\lambda,\nu} h_\nu,\] and using Lemma~\ref{lem:Zlambda-into- Schurs} again. 
\end{proof}
The following example illustrates Equation~\eqref{eqn:Phi-t-irreps}.
\begin{ex}\label{ex:Proof-prop:Phi-Jacobi-Trudi} 
     Let $\lambda=(3,2,1)$. The Jacobi-Trudi determinant expands to give 
     $s_\lambda= h_{321}-h_{33}-h_{411}+h_{51}$. 
     Applying $\Phi_t$, we obtain, for the multiplicity of the $t$-tuple of partitions $(\mu^1, \ldots, \mu^t)$ of 6  in the module with product Frobenius characteristic $\Phi_t(s_\lambda)$, the expression
     \[
     \langle \Phi_t(s_\lambda), \prod_{j=1}^ts_{\mu^j}(X^j)   \rangle=
     \prod_{j=1}^t K_{\mu^j, 321}-\prod_{j=1}^t K_{\mu^j, 33}- \prod_{j=1}^t K_{\mu^j, 411}  
     + \prod_{j=1}^t K_{\mu^j, 51}.\] 
\end{ex}

\begin{ex}\label{ex:small-reps-redux}  Computing $\beta_3^{(2)}$ using the above formulas gives 
$\beta_3^{(2)}=Z^{(2)}_{(3)}-2Z^{(2)}_{(2,1)}+Z^{(2)}_{(1^3)}$, since $e_3=h_3-2h_2h_1+h_1^3$.  Since $h_2h_1=s_{(3)}+s_{(2,1)}$ and $h_1^3=s_{(3)}+2s_{(2,1)}+s_{(1^3)}$, we obtain 
\begin{align*}\beta_3^{(2)}&=Z^{(2)}_{(3)}-2Z^{(2)}_{(2,1)}+Z^{(2)}_{(1^3)}\\
&= s_{(3)}(X^1)s_{(3)}(X^2) -2(s_{(3)}(X^1)+s_{(2,1)}(X^1))(s_{(3)}(X^2)+s_{(2,1)}(X^2))\\
& +(s_{(3)}(X^1)+2s_{(2,1)}(X^1)+s_{(1^3)}(X^1))  (s_{(3)}(X^2)+2s_{(2,1)}(X^2)+s_{(1^3)}(X^2)),
\end{align*}
which, upon simplification,  agrees with Example~\ref{ex:small-reps}.
\end{ex}

Item (1) generalises the result of E{\v g}ecio{\v g}lu and Remmel \cite[Corollary 1, Part (iv)]{EgeciogluRemmel1990} asserting that 
\begin{equation}\label{eqn:Eg-Remmel-en-to-h} e_n=\sum_{\lambda\vdash n} c_\lambda h_\lambda.\end{equation}
\begin{cor}\label{cor:irredS-beta-n-t}  Fix  $t\ge 1$.  Then the following hold.
\begin{enumerate}
    \item
For any fixed  $t$-tuple of partitions  $(\mu_1, \ldots, \mu_t)$ of $n$, the sum
\[ c^t_{\underline{\mu}}=\sum_{\lambda\vdash n} c_\lambda \prod_{j=1}^t K_{\mu^j, \lambda}
=\sum_{\lambda\vdash n} (-1)^{n-\ell(\lambda)}\frac{\ell(\lambda)!}{\prod_i m_i(\lambda)!} \prod_{j=1}^t K_{\mu^j, \lambda}\] 
is a  nonnegative integer.
\item  For $n\ge 2$, $\sum_{\lambda\vdash n}c_\lambda=0$.
\item $\sum_{\lambda\vdash n}c_\lambda\, \frac{n!}{\prod_{i\ge 1} \lambda_i! } =1$.
\end{enumerate}
\end{cor} 
\begin{proof}
The first part follows from Theorem~\ref{thm:Bnt-irreps}, Item (3).
      Items (2) and (3) are obtained from~\eqref{eqn:Eg-Remmel-en-to-h}, by respectively taking the multiplicity of the trivial representation, and dimensions. 
\end{proof}

The results of Corollary~\ref{cor:mult} can now be recovered from Item (1) of Theorem~\ref{thm:Bnt-irreps}, since $\langle h_\lambda, e_n\rangle= \delta_{\lambda, 1^n}$ and $\langle h_\lambda, h_n\rangle= 1$ for all $n.$

\begin{rem}\label{rem:Phit-Schur-true-module?} In contrast to Corollary~\ref{cor:irredS-beta-n-t}, the expressions in Item (4) of Theorem~\ref{thm:Bnt-irreps} 
may be negative integers; for arbitrary $\lambda$, 
the $\sym_n^{\times t}$-module whose product Frobenius characteristic is $\Phi_t(s_\lambda)$ may NOT be a true module. 
Using  SageMath reveals  the following counterexamples. 
  
  \begin{enumerate}
      \item  If $\lambda=322$, then for the 2-tuple 
$(\mu^1, \mu^2)=(43, 61)$, the multiplicity of $(\mu^1, \mu^2)$ in the module whose product Frobenius characteristic is $\Phi_2(s_{322})$ is $-1$.  Moreover, 
this implies that the module $\Phi_t(s_{322})$ does not correspond to a true module for any $t\ge 2$.
\item If $\lambda=2221$, the multiplicity of the 2-tuple 
$(\mu^1, \mu^2)=(331, 511)$ in the module corresponding to $\Phi_2(s_{2221})$ is $-2$.
  \end{enumerate}   
\end{rem}
However, we have the following.
\begin{prop}\label{prop:Phit-Schur-pos-cases} Let $\lambda\vdash n$.  Then 
$\Phi_t(h_\lambda)$ and $\Phi_t(e_\lambda)$ are the product Frobenius characteristics of  true $\sym_n^{\times t}$-modules.  

In addition, $\Phi_t(s_\lambda)$ is the product Frobenius characteristic of a true $\sym_n^{\times t}$-module if $\lambda$ has at most two parts.
\end{prop}
\begin{proof}
From Definition~\ref{def:map-Phi-S},   $\Phi_t(h_n)$ is the  characteristic of the trivial $\sym_n^{\times t}$-module. Since $\Phi_t$ is defined   multiplicatively, we have $\Phi_t(h_\lambda)=\prod_j \Phi_t(h_{\lambda_j})=Z_\lambda^{(t)}$, which  is the  characteristic of a true module.
 
 Theorem~\ref{thm:Bnt-irreps} tells us that $\Phi_t(e_n)$ is the product Frobenius characteristic $\beta_n^{(t)}$  of the top homology module of $B_n^{(t)}$.  Again we have $\Phi_t(e_\lambda)=\prod_j \Phi_t(e_{\lambda_j})=\prod_j \beta_{\lambda_j}^{(t)}$, and the result follows for $\Phi_t(e_\lambda)$.

 Now let $\lambda=(a,b)$ where $a\ge b\ge 1$. Consider the $t$-tuple $(\mu^1,\ldots, \mu^t)$ of partitions of $n$. The coefficient of $\prod_{j=1}^t s_{\mu^j} (X^j)$  in $\Phi_t(s_\lambda)$ is given by 
 \begin{equation}\label{eqn:coeff-2-parts}
 \prod_{j=1}^t K_{\mu^j, (a,b)} -\prod_{j=1}^t K_{\mu^j, (a+1, b-1)}. 
 \end{equation}
 We claim that for any partition $\mu$ of $a+b$, 
 $ K_{\mu, (a,b) }\ge K_{\mu, (a+1, b-1) }
 $. 

Recall that the Kostka number $K_{\mu,\nu}$ equals the inner product $\langle s_\mu, h_\nu\rangle$. Hence, by the Jacobi-Trudi identity, 
  \[K_{\mu, (a,b)} -K_{\mu, (a+1, b-1) } 
 =\langle s_\mu, h_a h_b- h_{a+1} h_{b-1}\rangle = \langle s_\mu, s_{(a,b)}\rangle,\]
 so that the difference 
 $ K_{\mu, (a,b)} -K_{\mu, (a+1, b-1)}$ is 1 if  $\mu=(a,b)$, and 0 otherwise.  The claim follows.
 
 This shows that, in \eqref{eqn:coeff-2-parts},  every factor in the second product  is less than or equal to the corresponding factor in the first product.
Consequently the coefficient~\eqref{eqn:coeff-2-parts} is nonnegative, finishing the proof.
\end{proof}

\begin{rem}\label{rem:dimS-beta-nt} Take dimensions in Item (3) of Theorem~\ref{thm:Bnt-irreps}.
We obtain the following curious enumerative identity.
\begin{equation}\label{eqn:dim-beta-nt-irrepS}
 w_n^{(t)}=\sum_{\lambda\vdash n}\sum_{\substack{\mu^i\vdash n\\1\le i\le t}} 
 (-1)^{n-\ell(\lambda)}\binom{\ell(\lambda)}{m_1(\lambda), m_(\lambda),\ldots} \,\prod_{j=1}^t K_{\mu^j, (1^n)} K_{\mu^j, \lambda}.
\end{equation}
\end{rem}

Table~\ref{tab:wn-t} is compiled from OEIS A212855.
\begin{table}[htbp]
\begin{center}
\begin{tabular}{|c|c|c|c|c|c|c|}
\hline
$t\backslash n$ & 0 & 1  &2 & 3 & 4 & 5   \\
[2pt]\hline
$ t=1$ & {1} & {1} & {  1} & {  1} & {  1} & {  1} \\
$2$ & {  1} & {  1} & {  3} & {  19} & {  211} & {  3651} \\
$3$ & 1 & {  1} & {  7} & {  163} & {  8983} & 966751 \\
$4$ & 1 & 1 & 15 & 1135 & 271375 &  \\
$5$ & 1 & 1 & 31 & 7291 & { 7225951} & \\
$6$ & 1 & 1 & 63 & 45199 & 182199871 &  \\
%
\hline
\end{tabular}
\end{center}
\vskip .1in
\caption{The numbers $w_n^{(t)}$ for $0\le n\le 5$, $1\le t\le 6$}
\label{tab:wn-t}
\end{table}
\subsection{The diagonal action of $\sym_n$ on $B_n^{(t)}$} 
Since the symmetric group $\sym_n$ itself acts diagonally on the $t$-fold Segre power $B_n^{(t)}$, we can ask for a description of this diagonal action.
 Recall the definition of 
$c^t_{\boldvec{\mu}}$ from Theorem~\ref{thm:Bnt-irreps}.  Also let $g^\lambda_{\boldvec{\mu}}$ denote the Kronecker coefficient \cite{Macd1995} 
$\langle  \chi^\lambda, \prod_{j=1}^t \chi^{\mu_j}\rangle$, 
that is, the multiplicity of the $\sym_n$-irreducible $\chi^\lambda$ in the tensor product of $\sym_n$-irreducibles $\chi^{\mu^j},$ 
$1\le j\le t$.  Finally recall \cite{Macd1995} that $*$ denotes the internal product in the ring of symmetric functions $\Lambda^n(X)$ in a single set of variables $X$.

From Theorem~\ref{thm:Bnt-irreps} we deduce the following. 
\begin{thm}\label{thm:Sn-diag-action-homology-Bn-t}
For the diagonal $\sym_n$-action on $\tilde{H}_{n-2}(B_n^{(t)})$:
\begin{enumerate}
\item
The (ordinary) Frobenius characteristic is this signed sum of characteristics of permutation modules:
\[\ch\, \tilde{H}_{n-2}(B_n^{(t)})=\sum_{\lambda\vdash n} c_\lambda \, (h_\lambda)^{*t},\]
writing 
$(h_\lambda)^{*t}$ for $\underbrace{h_\lambda*h_\lambda*\cdots *  h_\lambda}_t.$
\item
The multiplicity of the $\sym_n$-irreducible 
indexed by $\lambda$ is   %
$ \sum_{\boldvec{\mu}} c^t_{\boldvec{\mu}} g^\lambda_{\boldvec{\mu}},$
where the sum is over all $t$-tuples of partitions $\boldvec{\mu}=(\mu^1,\ldots,\mu^t)$,  $\mu^j\vdash n$, $1\le j\le t$.  
Equivalently,  
\[\ch\, \tilde{H}_{n-2}(B_n^{(t)})=\sum_{\boldvec{\mu}} c^t_{\boldvec{\mu}} g^\lambda_{\boldvec{\mu}}\, s_\lambda
=\sum_{\boldvec{\mu}} c^t_{\boldvec{\mu}}\, 
s_{\mu^1}*\cdots * s_{\mu^t}.\]
\item
The trace of an element $\sigma\in \sym_n$  is 
$\sum_{\lambda\vdash n} c_\lambda (\chi^{M^\lambda}(\sigma))^t,$ 
where $\chi^{M^\lambda} $ is the character of the permutation module corresponding to $h_\lambda$.

It is also equal to
$\sum_{\boldvec{\mu}}  c^t_{\boldvec{\mu}}\,\prod_{j=1}^t\chi^{\mu^j}(\sigma)$.
\end{enumerate}

\end{thm}
\begin{proof} The first item follows directly from Item (1) of Theorem~\ref{thm:Bnt-irreps}.

Item (3) of Theorem~\ref{thm:Bnt-irreps} gives us the following decomposition for the $\sym_{n}^{\times t}$-action on 
    $\tilde{H}_{n-2}(B_n^{(t)})$:
\[\sum_{\boldvec{\mu}} c^t_{\boldvec{\mu}} \bigotimes_{j=1}^t\chi^{\mu^j}.\]

The action of $\sym_n$ is obtained by restricting the $\sym_{n}^{\times t}$-action to its diagonal subgroup which is isomorphic to $\sym_n$, giving the following for the trace of $\sigma\in\sym_n$ on 
    $\tilde{H}_{n-2}(B_n^{(t)})$:
    \[\sum_{\boldvec{\mu}} c^t_{\boldvec{\mu}} \bigotimes_{j=1}^t\chi^{\mu^j}(\sigma).\] 
By definition of the internal product $*$ and Kronecker coefficients, 
$\ch(\bigotimes_{j=1}^t\chi^{\mu^j})=s_{\mu^1}*\cdots * s_{\mu^t} $.  The remaining items  now follow.
\end{proof}
\begin{rem}\label{rem:Segre-character-values-ordered-brick-tabloids}
By a theorem of E{\v g}ecio{\v g}lu and Remmel \cite[p. 109-111, Part (9)]{ER1991TransitionMatrices} (see also \cite[Transition Matrices]{PerAlexandersson}),  the coefficient of the power sum $p_\mu$ in the expansion of $h_\lambda$ is $|OB_{\mu, \lambda}|/z_\mu$, where $OB_{\mu, \lambda}$ counts the number of \emph{ordered $\mu$-brick tabloids of shape $\lambda$}.  We refer the reader to \cite[pp.108-111]{ER1991TransitionMatrices} for definitions.  It follows that the value of the character of $\chi^{M^\lambda}$ on the conjugacy class $\mu$ is the nonnegative integer 
$|OB_{\mu, \lambda}|$, giving the expression 
\begin{equation}\label{eqn:charvalue-diag-action-bricks}
\sum_{\lambda\vdash n} c_\lambda\, (|OB_{\mu, \lambda}|)^t
\end{equation}
for the character value of $\tilde{H}_{n-2}(B_n^{(t)})$ on the conjugacy class indexed by  $\mu$.

\end{rem}

Taking dimensions in Item (1) of Theorem~\ref{thm:Sn-diag-action-homology-Bn-t}, we obtain:
\begin{cor}\label{cor:dim-Bnt} For $t\ge 1$, 
$w_n^{(t)}=\sum_{\lambda\vdash n} c_\lambda \left(
\frac{n!}{\prod_{i\ge 1} \lambda_i!}\right)^t$.  
\end{cor}
\begin{prop}\label{prop:Sn-action-explicit-formulas}
   We have the following explicit formulas for $n=2$ and $n=3$.
\begin{enumerate}
\item  For $n=2$: $\tilde{H}_{0}(B_2^{(t)})= 2^{t-1} \chi^{(1^2)} + (2^{t-1}-1) \chi^{(2)},$
 in agreement with the case $t=1$. In fact it is $e$-positive:
 \[\ch\,\tilde{H}_{0}(B_2^{(t)})=(2^{t-1}-1)e_{(1,1)}+e_2.\]
 The dimension is $w_2^{(t)}=2^t-1$.
 \item For $n=3$:  
 $\tilde{H}_{1}(B_3^{(t)})= 2^{t-1} \chi^{(1^3)} + (2^{t-1}-1) \chi^{(3)} 
 +(2^{t}-2) \sum_{r=1}^t \binom{t}{r}(\chi^{(2,1)})^{\otimes r} $, again  agreeing with the case $t=1$, 
 and it is $e$-positive:
 \[\ch\,\tilde{H}_{1}(B_3^{(t)})= (6^{t-1}-3^{t-1}) e_{(1,1,1)} +e_3.\]
 The dimension is $w_3^{(t)}=6(6^{t-1}-3^{t-1})+1$.   (See $\mathtt{OEIS}$   A248225, A127222 for the sequence $\{6^t-3^t\}$.)
 %
 
 \item For $n=4$:
$\ch\,\tilde{H}_{2}(B_4^{(t)})=e_4-(2^{t-1}-1) e_{2,1,1} +(2^{t-1}-1) e_{2,2}  + \gamma_4\, e_{1,1,1,1}
$

where $\gamma_4(t)=\frac{(4^{t-1}-1)}{3}+ \frac{3(6^{t-2}-2^{t-2})}{2} -17\cdot 12^{t-2} +2^{t-2} +24^{t-1}.$  

The dimension is $w_4^{(t)}=1-6(2^{t-1}-1) +\gamma_4(t)$.
\end{enumerate} 
\end{prop}
\begin{proof}  We sketch the proof for Item (3), omitting details of the brute-force  computations involving internal products. The case $n=2$ is also easily computed directly, since $B_2^{(t)}$ is a rank 2  poset.

We compute the right-hand side of Item (1) of Theorem~\ref{thm:Sn-diag-action-homology-Bn-t}.
First observe that  the definition of $c_\lambda$ gives 
$c_{(4)}=-1, c_{(31)}=2, c_{(22)}=1, c_{(211)}=-3 \text{ and } c_{(1^4)}=1.$ 

We repeatedly use the fact that for $H$ a subgroup of a group $G$ and $V$, $W$ respectively $H$- and $G$-modules, one has the $G$-module isomorphism $V\uparrow_H^G\otimes W\simeq (V\otimes W\downarrow^G_H)\uparrow_H^G$.

This gives $(h_1^4)^{*t}= 24^t h_1^4.$  Also $(h_4)^{*t}=h_4.$

From \cite[Lemma 6.1]{SuSubword2021}, and also \cite[Ch. 7, Supplementary Problems 137 (a)]{ec2}, we have 
\[(h_{31})^{*t}=h_{31}+(2^{t-1}-1)h_{21^2} +(S(t,3)+S(t,4)) h_{1^4},\]
where $S(n,k)$ is the Stirling number of the second kind; also 
$S(t,3)+S(t,4)=\frac{4^{t-1}+2}{6}-2^{t-2}.$ 

An inductive argument also gives the following closed formulas.  
\begin{enumerate}
    \item $(h_{22})^{*t}=2^{t-1} h_{22}+(6^{t-2}+ \frac{6^{t-2}-2^{t-2}}{2}) h_{1^4}$.
    \item $(h_{21^2})^{*t}=2^{t-1}h_{21^2}+ (12^{t-2} \cdot 5 +2^{t-2}(6^{t-2}-1))  h_{1^4} $.
\end{enumerate}
Since 
\[\ch\,\tilde{H}_{2}(B_4^{(t)})=
c_{(4)} (h_4)^{*t}+c_{(31)}(h_{31})^{*t}+c_{(22)} (h_{22})^{*t} + c_{(21^2)} (h_{21^2})^{*t}+ c_{(1^4)} (h_{1^4})^{*t},\]
putting these facts together gives the stated expansion in the elementary symmetric functions.
\end{proof}
The first few values of $\gamma_4(t)$ are 
$0, 9, 375, 11309, 01085,7591669, 186637045.$ 

Sage data for $n, t\le 7$  supports the following conjectures for the diagonal action.
\begin{conj}\label{conj:diag-en} The coefficient of $e_n$ in the elementary basis expansion of $\ch\,\tilde{H}_{n-2}(B_n^{(t)})$ is always 1. 
\end{conj}

\begin{conj}\label{conj:diag-inclusions} For $t\ge1$, there are $\sym_n$-equivariant inclusions 
$\tilde{H}_{n-2}(B_n^{(t)})\hookrightarrow \tilde{H}_{n-2}(B_n^{(t+1)}).$
\end{conj}

The next conjecture may be viewed as a stability result for the set of irreducibles appearing in $\tilde{H}_{n-2}(B_n^{(t)})$.
\begin{conj}\label{conj:diag-inclusions-support} For $t\ge 3$, any irreducible appearing in $\tilde{H}_{n-2}(B_n^{(t)})$ also appears in 
$\tilde{H}_{n-2}(B_n^{(2)})$.
\end{conj}

\section{Rank-selection in $B_n^{(t)}$}\label{sec:rank-selection}

      For a fixed subset $J$ of the nontrivial ranks $[1,n-1]$, the rank-selected subposet $B_n^{(t)}(J)$ is defined to be the bounded poset $\{x\in B_n^{(t)}: \rk(x)\in J\}$ with the top and bottom elements $\hat 0, \hat 1$ appended.  Since rank-selection preserves the property of being Cohen-Macaulay \cite[Theorem~6.4]{Bac1980} (see also the references in \cite[Page 226]{SundaramAIM1994}),  these posets have at most one nonvanishing homology module, which is in the top dimension. The study of the homology of rank-selected subposets was initiated in \cite{RPSGaP1982}. 
 
 Denote by $\beta_n^{(t)}(J)$ the product Frobenius characteristic of the top homology $\tilde{H}_{k-1}(B_n(J))$ of the rank-selected subposet of $B_n^{(t)}$ corresponding to the rank-set $J$.
Our strategy to determine the rank-selected representations of $B_n^{(t)}$ will be to use the known results \cite{RPSGaP1982} for rank-selection in the Boolean lattice, and then apply the homomorphism $\Phi_t$ to obtain the corresponding results for the Segre product 
$B_n^{(t)}$.  

Stanley's theory of rank-selected invariants for Cohen-Macaulay posets, discussed in Section~\ref{sec:Segre-EL}, takes the following group-equivariant form, first described in \cite{RPSGaP1982}.  Note that dropping the Cohen-Macaulay condition  results in a similar formulation with the top homology module of each rank-selected subposet being replaced by the alternating sum of homology modules, i.e., the Lefschetz module (see \cite{Sheila1}).
\begin{thm}[{\cite{RPSGaP1982}}]\label{thm:RPS-GAP} Let $G$ be a finite group of automorphisms of a bounded and ranked Cohen-Macaulay poset $P$.  For each subset $J$ of the nontrivial ranks, let $P(J)$ denote the bounded rank-selected subposet $\{x\in P: \rk(x)\in J\}\cup\{\hat 0, \hat 1\}$. Then $G$ is a group of automorphisms of the Cohen-Macaulay subposet $P(J)$, and consequently the maximal chains in $P(J)$ and the top homology of $P(J)$ carry representations of $G$.  Let $\alpha_P(J)$ and $\beta_P(J)$ respectively denote these representations.
Then 
\begin{equation*}
\begin{split}
\alpha_P(J)&=\sum_{U\subseteq J} \beta_P(U),\\
 \beta_P(J)&=\sum_{U\subseteq J}(-1)^{|J|-|U|} \alpha_P(U).
\end{split}
\end{equation*}
\end{thm}

In the notation of Section~\ref{sec:Segre-EL},  the dimensions of the representations $\alpha_P(J)$, $\beta_P(J)$  are respectively $\tilde{\alpha}_P(J)$, $\tilde{\beta}_P(J)$.

Next we review the specific results for the Boolean lattice.
Let $J=\{1\le j_1<\cdots<j_r\le n-1\}$ be a subset of the nontrivial ranks $[n-1]$ of $B_n$. 
The $\sym_n$-module afforded by the maximal chains of $B_n(J)$ is easily seen to have Frobenius characteristic $\alpha_n(J)$ where 
\begin{equation}\label{eqn:chainsBn} \alpha_n(J)=h_{j_1} h_{j_2-j_1}\cdots h_{j_r-j_{r-1}}h_{n-j_r},
\end{equation}
since it is the permutation action on the cosets of the Young subgroup $\sym_{j_1}\times \sym_{j_2-j_1}\times\cdots\times \sym_{j_r-j_{r-1}}\times\sym_{n-j_r}$: the chains are permuted transitively by $\sym_n$, and the stabiliser of a chain is 
that Young subgroup.

Let $\beta_n(J)$ denote the homology of the rank-selected subposet $B_n(J)$ of the Boolean lattice $B_n$.  
For a standard Young tableau $\tau$ of shape $\lambda\vdash n$ (see \cite{ec2} for definitions), the descent set $\Des(\tau)$ of $\tau$ is the set of entries $i$ such that $i+1$ appears in a row strictly below $i$.  Using  Theorem~\ref{thm:RPS-GAP}, the  permutation module of 
Equation~\eqref{eqn:chainsBn},  and Robinson-Schensted insertion (see \cite{ec2}), 
Stanley  \cite[Section 4]{RPSGaP1982} now deduces the following result, originally due to Solomon. 
\begin{thm}[{\cite{SolJAlg1968, RPSGaP1982}}]\label{thm:RPS-ranks}  For any subset $J=\{1\le j_1<\cdots<j_r\le n-1\}$ of the nontrivial ranks of $B_n$, one has:
\begin{enumerate}
\item $\alpha_n(J)=\sum_{\lambda\vdash n} s_\lambda\ |\{ SYT\ \tau \text{ of shape  } \lambda:  \Des(\tau)\subseteq J\}|;$
\item $\beta_n(J)=\sum_{\lambda\vdash n} s_\lambda\ |\{ SYT\ \tau \text{ of shape  } \lambda:  \Des(\tau)=J\}|.$
\end{enumerate}
\end{thm}

Now consider the $\sym_n^{\times t}$-module $M_J$ of maximal chains in the rank-selected $t$-fold Segre power $B_n^{(t)}(J)$.  Denote its product Frobenius characteristic by $\alpha^{(t)}_n(J)$. Again it is clear that the chains are transitively permuted by $\sym_n^{\times t}$.  An element $x$ of the chain $c$ is of the form $x=(J_1, J_2, \ldots, J_t)$ where every $J_i$ has the same cardinality $n_x$ for some $n_x\in J$.  The stabiliser of the chain $c$ is thus 
$(\sym_{j_1}\times \sym_{j_2-j_1}\times\cdots\times \sym_{j_r-j_{r-1}}\times\sym_{n-j_r})^{\times t}$,
and hence, by definition of the product Frobenius characteristic map $\Pch$, we have 
\[\alpha^{(t)}_n(J)=\Pch\, M_J =\prod_{i=1}^t (h_{j_1} h_{j_2-j_1}\cdots h_{j_r-j_{r-1}}h_{n-j_r})(X^i).\]

Finally we obtain the analogue of Theorem~\ref{thm:RPS-ranks} for the $t$-fold Segre power $B_n^{(t)}$.  Recall that Theorem~\ref{thm:Bnt-irreps} gives  a possibly virtual  expression for the decomposition of $\Phi_t(s_\lambda)$ into irreducibles. However, by Theorem~\ref{thm:RPS-GAP},   
 $\alpha^{(t)}_n(J)$ and $\beta^{(t)}_n(J)$ are indeed the product Frobenius characteristics of true $\sym_n^{\times t}$-modules, namely the module of maximal chains of the rank-selected subposet $B_n^{(t)}(J)$ and the top homology module of $B_n^{(t)}(J)$. 

\begin{thm}\label{thm:tfoldSegreBn-ranks} For any subset $J$ of nontrivial ranks of 
$B_n^{(t)}$, the homomorphism $\Phi_t$ maps the Frobenius characteristic of the $\sym_n$-action on the chains of $B_n(J)$ to the product Frobenius characteristic of the $\sym_n^{\times t}$-action on the chains of $B_n^{(t)}(J)$.  More precisely,  we have:
\begin{enumerate}
\item $\alpha^{(t)}_n(J)=\Phi_t(\alpha_n(J))=\sum_{\lambda\vdash n} \Phi_t(s_\lambda)\ |\{ SYT\ \tau \text{ of shape  } \lambda:  \Des(\tau)\subseteq J\}|;$
\item $\beta^{(t)}_n(J)=\Phi_t(\beta_n(J))=\sum_{\lambda\vdash n} \Phi_t(s_\lambda)\ |\{ SYT\ \tau \text{ of shape  } \lambda:  \Des(\tau)=J\}|.$
\end{enumerate}
\end{thm}

\begin{proof} It follows immediately from the preceding discussion, by applying the homomorphism $\Phi_t$ to \eqref{eqn:chainsBn}, that 
\[\alpha^{(t)}_n(J)=\prod_{i=1}^t (h_{j_1} h_{j_2-j_1}\cdots h_{j_r-j_{r-1}}h_{n-j_r})(X^i)=\Phi_t(\alpha_n(J)).\]
Hence from Theorem~\ref{thm:RPS-GAP} we obtain 
\[\beta^{(t)}_n(J)=\sum_{U\subseteq J}\alpha^{(t)}_n(U)=\sum_{U\subseteq J}\Phi_t(\alpha_n(U))=\Phi_t(\beta_n(J)).\]
Invoking Theorem~\ref{thm:RPS-ranks} now finishes the proof.\qedhere
\end{proof}

A skew-shape is said to be  \emph{connected}  \cite[p. 345]{ec2} if every two consecutive rows in its Ferrers diagram overlap in at least one square.  
Recall that a ribbon (also called a  border strip \cite{Macd1995, ec2}, or rim hook)  is a connected skew-shape with no 2 by 2 square. Equivalently, a ribbon is a skew-shape where two consecutive rows overlap in exactly one square. 

It follows from  Solomon's result \cite{SolJAlg1968}  that if $J=\{1\le j_1<j_2<\cdots<j_r\le n-1\}$, then $\beta_n(J)$ is the (Frobenius characteristic of the) $\sym_n$-module indexed by the  ribbon  whose Ferrers diagram has rows of lengths $j_1, j_2-j_1, \ldots , n-j_r$,  top to bottom. 
  In particular when $J$ consists of $k$ consecutive ranks starting at the first rank, $J=[k]$, then $\beta_n(J)$ corresponds to the $\sym_n$-irreducible indexed by the hook shape $(n-k, 1^k)$.

We can now make the following addition to Proposition~\ref{prop:Phit-Schur-pos-cases}.

\begin{cor}\label{cor:Phi-t-Schur-pos2} 
Let $f=s_{(n-k, 1^k)}$ be the Schur function indexed by a hook $(n-k, 1^k)$, or more generally the  Schur function indexed by a ribbon corresponding to the set $J=\{1\le j_1<j_2<\cdots<j_r\le n-1\}$, that is, whose rows top to bottom are $j_1, j_2-j_1, \ldots, n-j_r$.  Then $\Phi_t(f)$ is the product Frobenius characteristic of a true $\sym_n^{\times t}$-module.
\end{cor}
\begin{proof} This is immediate from the preceding remarks and the fact that $\Phi_t(\beta_n(J))=\beta_n^{(t)}(J)$. 
\end{proof}

Theorem~\ref{thm:tfoldSegreBn-ranks} can also be deduced by explicitly deriving a recurrence for the rank-selected homology, 
  using the Whitney homology technique  of \cite{SundaramAIM1994}. Indeed,  for the Boolean lattice itself (cf. \cite[Theorem~1.10, Example~1.11]{SundaramAIM1994}), if $J=\{1\le j_1<\cdots < j_r\le n-1\}$ is a subset of nontrivial ranks in $B_n$,  we have that 
\begin{align*}
\beta_n(J) +\beta_n(J\setminus\{j_r\})&=\ch \Whit_{j_r}(B_n(J))\\
&=\ch \bigoplus_{x:|x|=j_r}\tilde{H}(\hat 0, x)_{B_n(J)}.
\end{align*}
Since the interval $(\hat 0, x)_{B_n(J)}$ is isomorphic to the rank-selected subposet $B_{j_r}(J\setminus\{j_r\}) $, checking stabilisers gives  the recurrence 
\begin{equation}\label{eqn:rank-select-Bn-homology}
\beta_n(J) +\beta_n(J\setminus\{j_r\})=
\beta_{j_r}(J\setminus\{j_r\})\, h_{n-j_r}.
\end{equation}

Similarly, for the rank-selected homology representation $\beta^{(t)}_n(J)$ of the $t$-fold Segre power $B_n^{(t)}$, we obtain, by an analysis completely analogous to the above, and the one carried out in the proof of Theorem~\ref{thm:def-rec-tfoldSegreBn}, or alternatively by simply applying the homomorphism $\Phi_t$ to~\eqref{eqn:rank-select-Bn-homology}:

\begin{thm}\label{thm:rank-selected-homology-Bnt} Let $J=\{1\le j_1<\cdots < j_r\le n-1\}$ be a subset of nontrivial ranks in $B_n^{(t)}$.  The product Frobenius characteristic $\Pch\, \tilde{H}(B_n^{(t)}(J))$ of the rank-selected subposet of $B_n^{(t)}$ satisfies the following recurrence.
\begin{equation*}
\beta^{(t)}_n(J) +\beta^{(t)}_n(J\setminus\{j_r\})=
\beta^{(t)}_{j_r}(J\setminus\{j_r\})\, \prod_{i=1}^t h_{n-j_r}(X^i).
\end{equation*}
\end{thm}

This is precisely the $\sym_n^{\times t}$-equivariant version of ~\eqref{eqn:actual-rec-mu-rank-select-Bnq} of Proposition~\ref{prop:rec-rank-select-Betti-Bnq}, a connection that will be exploited in the next section.

\section{The $t$-fold  Segre power of  the  subspace lattice: rank-selection and  stable principal specialisation}\label{sec:stable-ps}

 We will show  in this section that the surprising relationship discovered in \cite{YLiqCSV2023} between the homology representation of the Boolean Segre square  $B_n\circ B_n$ and the M\"obius number of the subspace lattice Segre square  $B_{n,q}\circ B_{n,q}$ holds in greater generality, namely for all rank-selected subposets of the $t$-fold Segre powers in each case.

Recall that $\beta_n^{(t)}:=\Pch(\tilde{H}_{n-2}(B_n^{(t)}))$, where $B_n^{(t)}$ is the $t$-fold Segre power of $B_n$. The recurrence for $\beta_n^{(t)}$ is 
\begin{equation}\label{eqn:rec-repn-tfold-subsetlattice} 
\beta^{(t)}_n=\sum_{i=0}^{n-1} (-1)^{n-1-i}\ \beta^{(t)}_i\prod_{i=1}^t h_{n-i}(X^i)  
\end{equation}
The \emph{stable principal specialisation} \cite[Chapter 7, Section 8]{ec2} of a symmetric function $f$ in variables $x_1,x_2,\ldots$ is the function of $q$ obtained from $f$ by means of the substitution $x_i\rightarrow q^{i-1}, i\ge 1$.  Similarly, we consider the stable principal specialisation of a function in $\otimes_{i=1}^t \Lambda(X^i)$ to be the function of $q$ obtained by replacing each set of variables $X^i$ by the set $\{1,q,q^2,\ldots \}$. Using the stable principal specialisations $\ps \beta_n^{(t)}$ and $\ps h_n:=h_n(1,q,q^2,\dots )=\prod_{i=1}^n (1-q^i)^{-1}$ \cite[Proposition 7.8.3]{ec2}, from ~\eqref{eqn:rec-repn-tfold-subsetlattice} we have the following identity for $t\ge 1$. 
\begin{equation}\label{eqn:rec-qps-repn-tfold-subsetlattice} 
\ps\beta^{(t)}_n=\sum_{i=0}^{n-1} (-1)^{n-1-i} \left(  \prod_{j=1}^{n-i} (1-q^j)^{-1}\right)^t \ps\beta^{(t)}_i
\end{equation}

The extension of \cite[Theorem~4.2]{YLiqCSV2023} to $t$-fold products is now immediate.  Note the particular case $t=1$, for which $W_n^{(1)}(q)=q^{\binom{n}{2}}$ and $\ps e_n= 
\frac{q^{\binom{n}{2}}}{\prod_{j=1}^n(1-q^j)}$.
\begin{prop}\label{prop:qps-equals-mu-qBn} We have, for $t\ge 1$, 
\[\ps \Pch(\tilde{H}_{n-2}(B_n^{(t)}))=\ps\beta^{(t)}_n=\frac{W^{(t)}_n(q)}{\prod_{j=1}^n(1-q^j)^t}.\]
\end{prop}
    
\begin{proof} 
From~\eqref{eqn:q-bin-REMOVE}, we obtain 
\[{n\brack i}_q\cdot \prod_{j=i+1}^n (1-q^j)^{-1}= \prod_{j=1}^{n-i} (1-q^j)^{-1} =\ps h_{n-i}. \]
Thus, dividing~\eqref{eqn:rec-mu-tfold-subspacelattice}
throughout by $\prod_{j=1}^n(1-q^j)^t$, we obtain 
\[\frac{W^{(t)}_n(q)}{\prod_{j=1}^n(1-q^j)^t}=\sum_{i=0}^{n-1} (-1)^{n-1-i} \left(\prod_{j=1}^{n-i} (1-q^j)^{-1}\right)^t\frac{W^{(t)}_i(q)}{\prod_{j=1}^i (1-q^j)^t},\]
showing that $\frac{W^{(t)}_n(q)}{\prod_{j=1}^n(1-q^j)^t}$ and $\ps\beta_n^{(t)}$ satisfy the same recurrence.  
They also satisfy the same initial conditions, since $\beta_0^{(t)}=1$ 
implies $\ps \beta_0^{(t)}=1=W_0^{(t)}(q), $ and $\beta_1^{(t)}=h_1^t$ 
implies $\ps \beta_1^{(t)}=\frac{1}{(1-q)^t}=\frac{W_1^{(t)}(q)}{(1-q)^t}.$  This completes the proof.
    \end{proof}

Now we turn to rank-selection.   Recall that we denote by $\beta_n^{(t)}(J)$ the product Frobenius characteristic of the top homology $\tilde{H}_{k-1}(B_n^{(t)}(J))$ of the rank-selected subposet of $B_n^{(t)}$ corresponding to the rank-set $J=\{1\le j_1<\cdots<j_r\le n-1\}$.  Similarly, for the $t$-fold Segre power of the subspace lattice $B_{n,q}$, in Section~\ref{sec:Segre-EL} we denoted by $\tilde\beta_{B_{n,q}^{(t)}}(J)$ 
its rank-selected Betti number, i.e., the dimension of the top homology module of the rank-selected subposet  $B_{n,q}^{(t)}(J)$ corresponding to the rank-set $S$. Thus $\tilde\beta_{B_{n,q}^{(t)}}(J)$
is the absolute value of the M\"obius number of $B_{n,q}^{(t)}(J)$.  In particular, 
from Theorem~\ref{thm:rank-sel-beta-inv-N-of-q},
  $W^{(t)}_n(q)=\tilde\beta_{B_{n,q}^{(t)}}([n-1])$.
Using the recurrences in Theorem~\ref{thm:rank-selected-homology-Bnt} and Equation~\eqref{eqn:actual-rec-mu-rank-select-Bnq},  we can now derive the rank-selected analogue of Proposition~\ref{prop:qps-equals-mu-qBn}.

\begin{thm}\label{thm:ps-rank-selection} The stable principal specialisation of $\beta_n^{(t)}(J)$ for the rank-selected homology module of the $t$-fold Segre power of the Boolean lattice $B_n$, and the rank-selected Betti number  $\tilde\beta_{B_{n,q}^{(t)}}(J)$
of the $t$-fold Segre power of the subspace lattice $B_{n,q}$, are related by the equation
\begin{equation*} \ps \beta_n^{(t)}(J)= 
\frac{\tilde\beta_{B_{n,q}^{(t)}}(J)}{\prod_{i=1}^n(1-q^i)^t}.
\end{equation*}
\end{thm}
\begin{proof}
We  verify that  each side of the above equation satisfies the same recurrence. For a fixed subset $J$ of the nontrivial ranks $[1,n-1]$, the rank-selected subposet $B_n^{(t)}(J)$ is defined to be the bounded poset $\{x\in B_n^{(t)}: \rk(x)\in J\}$ with the top and bottom elements $\hat 0, \hat 1$ appended. 
First, taking the principal specialisation in Theorem~\ref{thm:rank-selected-homology-Bnt} gives us the following recurrence for the left-hand side, $\ps \beta_n^{(t)}(J)$. 

\begin{equation}\label{eqn:ps-rank-selected-homology-Bnt} 
\ps \beta_n^{(t)}(J) + \ps \beta_n^{(t)}(J\setminus\{j_r\}) 
=\ps \beta_{j_r}^{(t)}(J\setminus\{j_r\})\cdot (\ps h_{n-j_r})^t.
\end{equation}

For the right-hand side of the statement,  use the recurrence for $\tilde\beta_{B_{n,q}^{(t)}}(J)$ in 
Equation~\eqref{eqn:actual-rec-mu-rank-select-Bnq}, which  may be restated in the equivalent form 
\begin{equation}\label{eqn:rec-mu-rank-select-Bnq} 
\tilde\beta_{B_{n,q}^{(t)}}(J) +\tilde\beta_{B_{n,q}^{(t)}}(J\setminus\{j_r\})
 = \tilde\beta_{B_{j_r}^{(t)}(q)}(J\setminus\{j_r\}  )\cdot (\ps h_{n-j_r})^t\,\cdot \prod_{i=1+j_r}^n(1-q^i)^t.
\end{equation}

Dividing  \eqref{eqn:rec-mu-rank-select-Bnq}  by $\prod_{i=1}^n (1-q^i)^t$  and comparing with the recurrence \eqref{eqn:ps-rank-selected-homology-Bnt} for the principal specialisation     immediately shows that  the expressions 
$\ps \beta^{(t)}_n(J)$ and $\frac{\tilde\beta_{B_{n,q}^{(t)}}(J)}{\prod_{i=1}^n (1-q^i)^t}$ 
satisfy the same recurrence on subsets $J\subseteq [n-1]$.  Note also that the theorem has been established when $J=[n-1]$. 

Now consider the case when $J$ consists of a single rank $\{r\}$. Then we have 
\[\beta^{(t)}_n(J)=\prod_{j=1}^t h_{r}(X^j) \prod_{j=1}^t h_{n-r}(X^j) 
- \prod_{j=1}^t h_{n}(X^j), 
\]
\begin{equation*}
\begin{split}
\text{ and so }\quad \ps \beta^{(t)}_n(J)&=\left(\prod_{i=1}^r (1-q^i)^{-1} \prod_{i=1}^{n-r} (1-q^i)^{-1}\right)^t
- (\prod_{i=1}^n (1-q^i)^{-1})^t\\
&=\prod_{i=1}^n (1-q^i)^{-t} \left({n\brack r}_q^t-1\right)\\
&=\frac{\tilde\beta_{B_{n,q}^{(t)}}(J)}{\prod_{i=1}^n(1-q^i)^t}
.
\end{split}
\end{equation*}
An inductive argument on the size of $J$ now completes the proof.
\end{proof}

\section{Further questions}\label{sec:End}

In a future paper we examine the diagonal $\sym_n$-action on the (rank-selected) homology of the $t$-fold Segre power of $B_n$ (Theorem~\ref{thm:Sn-diag-action-homology-Bn-t}) more closely, investigating the conjectures at the end of Section~\ref{sec:tfold-Boolean-lattice-repn}.  It would also be interesting to see what more can be said about the map $\Phi_t$.

In \cite{RPSGaP1982}, Stanley examined the rank-selected homology of the subspace lattice. A logical  next step is to carry out the program of this paper for the action of $GL_n(q)$ on the Segre powers of the subspace lattice.


\end{document}